\documentclass[reqno,a4paper]{amsart}
\usepackage{eucal,amsfonts,amssymb,amsmath,amsthm,epsfig,mathrsfs}

\usepackage{amscd,amsxtra}
\usepackage{enumerate}
\usepackage{latexsym}

\allowdisplaybreaks

 \makeatletter \@addtoreset{equation}{section}

\makeatother \makeatletter

\newtheorem{thm}{Theorem}[section]
\newtheorem{hyp}[thm]{Hypothesis}{\rm}
\newtheorem{lemm}[thm]{Lemma}

\newtheorem{prop}[thm]{Proposition}
\newtheorem{defi}[thm]{Definition}
\newtheorem{rmk}[thm]{Remark}{\rm}

\newcommand{\R}{{\mathbb R}}
\newcommand{\N}{{\mathbb N}}

\newcommand{\Rd}{\mathbb R^d}

\newcommand{\A}{\mathcal{A}}

\newcommand{\eps}{\varepsilon}

\newcommand{\ds}{\displaystyle}

\newcommand{\bd}{\begin{defi}}
\newcommand{\ed}{\end{defi}}
\newcommand{\nnm}{\nonumber}
\newcommand{\be}{\begin{equation}}
\newcommand{\ee}{\end{equation}}
\newcommand{\barr}{\begin{array}}
\newcommand{\earr}{\end{array}}
\newcommand{\bmn}{\begin{eqnarray}}
\newcommand{\emn}{\end{eqnarray}}
\newcommand{\bnm}{\begin{eqnarray*}}
\newcommand{\enm}{\end{eqnarray*}}
\newcommand{\bln}{\begin{subequations}}
\newcommand{\eln}{\end{subequations}}

\newcommand{\ba}{\begin{align}}
\newcommand{\ea}{\end{align}}
\newcommand{\banm}{\begin{align*}}
\newcommand{\eanm}{\end{align*}}

\title[Semilinear non autonomous equations]{Semilinear nonautonomous parabolic equations with unbounded coefficients in the linear part}
\author{L. Angiuli}
\author{A. Lunardi}
\address{Dipartimento di Matematica, Universit\`a del Salento, via per Arnesano, s.n., 73100 Lecce, Italy.}
\address{Dipartimento di Matematica, Universit\`a degli Studi di Parma, Parco Area delle Scienze 53/A, I-43124 Parma, Italy.}
\email{luciana.angiuli@unsalento.it}
\email{alessandra.lunardi@unipr.it}

\date{\today}
\begin{document}

\begin{abstract}
We study the Cauchy problem for the semilinear nonautonomous parabolic equation $u_t=\mathcal{A}(t)u+\psi(t,u)$ in $[s,\tau]\times \Rd$, $\tau> s $,  in the spaces $C_b([s, \tau]\times\Rd)$ and in $L^p((s, \tau)\times\Rd, \nu)$. Here $\nu$ is a Borel measure defined via a tight evolution system of measures for the evolution operator $G(t,s)$ associated to the family of time depending second order uniformly elliptic operators $\A(t)$. Sufficient conditions for existence in the large and stability of the null solution are also given in both $C_b$ and $L^p$ contexts. The novelty with respect to the literature is that the coefficients of the operators $\mathcal{A}(t)$ are allowed to be unbounded.
\end{abstract}

\maketitle

\section{Introduction}

This paper is devoted to the basic theory of a class of semilinear nonautonomous parabolic problems with non standard linear part. We consider Cauchy problems such as
\begin{equation}\label{C_P_intro}
\left\{
\begin{array}{l}
 D_t u(t,x)  = ({\mathcal{A}}(t)u)(t,x) +\psi(t,u(t,x)),\qquad\;\, t>s, \;x\in  \Rd,
\\
\\
u(s,x) =  f(x),\quad  x \in \Rd,
\end{array}\right.
\end{equation}
where  the elliptic operators
$${\mathcal{A}}(t) := \sum_{i,j=1}^d q_{ij}(t,x)D_{ij} + \sum_{i=1}^d b_i(t,x)D_i $$
have unbounded  coefficients $q_{ij}$, $b_i$ in $I\times \R^d$, $I$ being a right halfline or the whole $\R$, $D_i= \partial/\partial x_i$, $D_{ij}= \partial^2/\partial x_i\partial x_j$. To our knowledge, no result for this type of problems is available in the literature.

We make suitable assumptions on the coefficients in order that the linear part generates a Markov evolution operator $G(t,s)$ in $C_b(\R^d)$, the space of the bounded and continuous functions from $\R^d$ to $\R$. The coefficients of $\A(t)$ are smooth enough, namely locally $C^{\alpha/2, \alpha}$ for some $\alpha\in (0,1)$, the matrices $Q(t,x) = [q_{ij}(t,x)]_{i,j=1, \ldots d}$ are uniformly positive definite, and there exists a $C^2$ Lyapunov function  $\varphi:\R^d\mapsto [0, +\infty)$ such that
$$\lim_{|x|\to +\infty}\varphi(x)=+\infty,  \quad
(\mathcal{A}(t)\varphi)(x)\leq a-c\,\varphi(x),  \quad (t,x)\in I\times \Rd , $$
for some positive constants $a$ and $c$. Such assumption allows to use maximum principle arguments both in linear and in nonlinear equations; see e.g. the proof of Theorem \ref{Exlarge}. The evolution operator $G(t,s)$ is a contraction in $C_b(\R^d)$, namely
$$\|G(t,s)f\|_{\infty}\leq \|f\|_{\infty}, \quad f\in C_b(\R^d), $$
and for any $s\in I$, $(t,x)\mapsto (G(t,s)f)(x)\in C^{1,2}((s, +\infty)\times \R^d)\cap C ([s, +\infty)\times \R^d)$ is the unique bounded solution of
$$\left\{\begin{array}{l}
D_t v(t,x)=(\mathcal{A}(t)v)(t,x)  ,\quad  t>s, \;x\in  \Rd,\\
\\
v(s,x)= f(x),\quad  x \in \Rd .
\end{array} \right. $$
uniformly for $s<t$ in bounded intervals, say $a\leq s<t\leq b$.
The construction of the evolution operator $G(t,s)$ and its main properties are in \cite{KunLorLun09Non}.

By ``solution" to \eqref{C_P_intro} in an interval $[s, \tau]$ we mean a mild solution, namely a function that satisfies the identity
$$u(t, \cdot) = G(t,s)f + \int_s^t G(t,r)\psi(r, u(r, \cdot))dr, \quad s\leq t\leq \tau. $$

If  $f\in C_b(\R^d)$, the usual arguments for parabolic equations with standard linear part (e.g. \cite{LadSolUra68Lin,Hen81Geo}) are adapted to the present situation and lead to existence and uniqueness of a local mild solution, which is shown to be a classical solution under reasonable assumptions. To this aim  we prove regularity and asymptotic behavior results for mild solutions of linear nonhomogeneous Cauchy problems,
$$u(t, \cdot) = G(t,s)f + \int_s^t G(t,r)g(r,   \cdot)dr, \quad s\leq t\leq \tau. $$
While the case $g\equiv 0$ was thoroughly studied in \cite{KunLorLun09Non,AngLorLun},   the nonhomogeneous case was neglected. Here we prove local and global regularity results in Section 2 and an asymptotic behavior result in Section 4, that are used as tools in the nonlinear case.

The case of $L^p$ initial data is more difficult. Even in the linear autonomous case $\mathcal{A}(t)\equiv  \mathcal{A}$, the Cauchy problem may be not well posed in $L^p(\R^d, dx)$ if the coefficients of $\mathcal{A}$ are unbounded, unless the coefficients satisfy very restrictive growth assumptions. The only way to work in $L^p$ spaces is to replace the Lebesgue measure $dx$ by another measure, possibly a   weighted measure $\rho(x)dx$. The best situation in the autonomous case   is when there exists an invariant measure $\mu $, namely a Borel probability measure such that
$$\int_{\R^d}T(t)f\,d\mu = \int_{\R^d} f\,d\mu , \quad t>0, \;f\in C_b(\R^d), $$
where $T(t)$ is the Markov semigroup associated to $\mathcal{A}$ in $C_b(\Rd)$. Under reasonable assumptions, a unique invariant measure exists, it is absolutely continuous with respect to the Lebesgue measure,
and it is related to the asymptotic behavior of $T(t)$, since
$$\lim_{t\to +\infty} (T(t)f)(x) = \int_{\R^d} f\,d\mu, \quad f\in C_b(\R^d), \;x\in \R^d. $$
Moreover, the operators $T(t)$ are easily  extended to contractions in the spaces $L^p(\R^d, \mu)$ for every $p\in [1, +\infty)$.

The nonautonomous case is more complex. In general, a measure $\mu$ such that
$$\int_{\R^d}G(t,s)f\,d\mu = \int_{\R^d} f\,d\mu , \quad t>s, \;f\in C_b(\R^d), $$
does not exist. What plays the role of  invariant measures are the evolution systems of measures, namely families of Borel probability measures $\{ \mu_t:\;t\in I\}$ such that
$$\int_{\R^d}G(t,s)f\,d\mu_t = \int_{\R^d} f\,d\mu_s , \quad s\in I, \;t>s, \;f\in C_b(\R^d). $$
In this case, $G(t,s)$ can be extended to a contraction from $L^p(\R^d, \mu_s)$ to $L^p(\R^d, \mu_t)$ for $t>s$, for every $p\in [1, +\infty)$. However, in contrast to the autonomous case, where the invariant measure is unique under very weak assumptions, evolution systems of measures are not unique. Among all evolution systems of measures, the one related to the asymptotic behavior of $G(t,s)$ is the (unique) tight$^(\footnote{A set of Borel measures $\{\mu_t:\;t\in I\}$ in $\R^d$ is tight if for every $\varepsilon >0$ there exists $\rho>0$ such that $\mu_t(\R^d\setminus B(0, \rho))\leq \varepsilon$, for every $t\in I$. }^)$ evolution system of measures. See \cite{KunLorLun09Non,AngLorLun}.

In the paper \cite{KunLorLun09Non} a tight evolution system of measures $\{ \mu_t:\;t\in I\}$ was proved to exist. Here we set our nonlinear problem in the spaces $L^p(\R^d, \mu_t)$ where $\{ \mu_t:\;t\in I\}$ is such a tight evolution system of measures. As usual, to work in a $L^p$ context the nonlinearity is assumed to be Lipschitz continuous with respect to $u$. We introduce the measure $\nu$ in $I\times \R^d$, defined by
$$\nu(J\times \mathcal O):= \int_J\mu_t(\mathcal O)\,dt,$$
on Borel sets $J\subset I$, $\mathcal{O}\subset \Rd$ and canonically extended to the Borel sets of $I\times \R^{d}$. For every $f\in L^p(\R^d, \mu_s)$ we prove existence in the large and uniqueness of a solution $u$ to \eqref{C_P_intro} belonging to  $L^p((s, \tau)\times \R^d, \nu)$, for every $\tau >s$. Moreover,   $\sup_{s<t<\tau}\|u(t, \cdot)\|_{L^p(\R^d, \mu_t)} <\infty$.

Note that \eqref{C_P_intro} cannot be seen as an evolution equation in a fixed $L^p$ space, because our spaces $L^p(\R^d, \mu_t)$ may depend explicitly on $p$.

In Section 5 we turn to global estimates, asymptotic behavior and summability improving results. Assuming that $\psi(t,0)=0$ for every $t$, we prove a nonautonomous version of the principle of linearized stability in the space $C_b(\R^d)$. In addition, under a   dissipativity assumption on $\psi$,
$$\xi\,\psi(t,\xi)\le \psi_0\, \xi^2,\qquad\;\, t \in I,\,\xi \in \R,$$
with $\psi_0\in \R$, we prove that for every $f\in C_b(\R^d)$, the solution $u$ to \eqref{C_P_intro} satisfies
\begin{equation}\label{parma}
|u(t,x)|\leq e^{\psi_0(t-s)}\|f\|_{\infty}, \quad t>s, \;x\in \R^d.
\end{equation}
So, the null solution is globally stable if $\psi_0=0$, exponentially globally stable if $\psi_0<0$. The same assumption, together with some technical assumptions on the growth of the coefficients as $|x|\to \infty$, allows to prove a similar result in our $L^p$ context:
 for every $f\in L^p(\R^d, \mu_s)$, the solution $u$ to \eqref{C_P_intro} satisfies
\begin{equation}\label{sabato}
\|u(t,\cdot )\|_{L^p(\R^d, \mu_t)}\leq e^{\psi_0(t-s)}\|f\|_{L^p(\R^d, \mu_s)}, \quad t>s.
\end{equation}
If the measures $\mu_t$ satisfy a uniform logarithmic Sobolev type inequality with constant $K$,
\begin{equation}\label{LogSob}
\int_{\Rd} |g|^\gamma\log |g| \,d\mu_r\le  \|g\|_{L^\gamma(\Rd,\mu_r)}^\gamma \log\|g\|_{L^\gamma(\Rd,\mu_r)} + \gamma K\int_{\{g\neq 0\}}\!\!\!\!|g|^{\gamma-2}|\nabla g|^2 d\mu_r,
\end{equation}
for any $r\in I$, $g\in C^1_b(\R^d)$ and $\gamma\in (1,+\infty)$, then estimate \eqref{sabato} can be improved as follows,
\begin{equation}\label{parma1}
\|u(t,\cdot )\|_{L^{p(t)}(\R^d, \mu_t)}\leq e^{\psi_0(t-s)}\|f\|_{L^p(\R^d, \mu_s)},\quad t>s,
\end{equation}
where  $p(t) := e^{\eta_0K^{-1}(t-s)}(p-1)+1$, $\eta_0$ being the ellipticity constant.   So, we get a hypercontractivity property that is similar to the linear case (\cite{AngLorLun}) if $\psi_0=0$, hypercontractivity plus exponential decay if $\psi_0<0$.

Note that estimates \eqref{parma}, \eqref{sabato} and \eqref{parma1} are significant also if $\psi_0>0$.

Several examples of operators ${\mathcal A}(t)$ that satisfy our assumptions are in the papers \cite{KunLorLun09Non,AngLorLun} to which we refer for detailed proofs. In particular, we allow for time dependent Ornstein--Uhlenbeck operators
$${\mathcal A}(t)\zeta(x) =  \sum_{i,j=1}^d q_{ij}(t)D_{ij}\zeta(x) + \sum_{i, j=1}^d (b_{ij}(t)x_j+f_i(t))D_i \zeta(x)$$
with bounded and locally H\"older continuous $q_{ij}$, $b_i$, $f_i$ and uniformly elliptic diffusion part. In this case, we can take $\varphi(x) = |x|^2 $ if the matrices
$[b_{ij}(t)]_{i,j=1, \ldots ,d}$ are uniformly negative definite; the tight evolution system of measures is explicit and it consists of suitable Gaussian measures depending on $t$. See \cite{GeisLun08}, where the evolution operator $G(t,s)$ for nonautonomous  Ornstein--Uhlenbeck equations and the associated evolution systems of measures were studied under weaker assumptions than the present ones. This is the only nontrivial case such that the measures $\mu_t$ are explicitly known. In the other cases, several properties of the measures $\mu_t$ were proved in the above mentioned papers \cite{KunLorLun09Non,AngLorLun} and in \cite{LorLunZam10}, that dealt with  the time periodic case $q_{ij}(t+T, x)= q_{ij}(t, x)$, $b_{i}(t+T, x)= b_{i}(t, x)$. In that case,  the tight evolution system of measures is also $T$-periodic. Sufficient conditions for the occurrence of the logarithmic Sobolev  inequalities \eqref{LogSob} are in \cite{AngLorLun}.

\subsection*{Notations}

For $k\ge 0$, $d \ge 1$, by $C^k_b(\Rd)$ we mean the space of the functions  in $C^k(\Rd)$ which are bounded
together with all their derivatives up to the $[k]$-th order.
$C^k_b(\Rd)$ is endowed  with the norm $\|f\|_{C_b^k(\Rd)}=\sum_{|\alpha|\le k}\|D^\alpha f\|_{\infty}+\sum_{|\alpha|=[k]}|D^\alpha f|_{C_b^{k-[k]}(\Rd)}$ where $\|\cdot\|_\infty$ and $[k]$ denote respectively the sup-norm and the integer part of $k$.
When $k\notin \N$, we use the subscript ``loc'' to denote the space of all $f\in C^{[k]}(\Rd)$ such that $D^{[k]}f$ is $(k-[k])$- H\"older continuous in any compact subset of $\Rd$.
The space of bounded Lipschitz continuous functions is denoted by ${\rm Lip}_b(\Rd)$ and equipped with the norm $\|\cdot\|_{\infty}+|\cdot|_{{\rm Lip}(\Rd)}$.

For any interval $J\subset \R$,  $C^{\alpha/2,\alpha}(J\times \Rd)$ ($\alpha\in (0,1)$) denotes the usual parabolic H\"older space and the
subscript ``loc'' has the same meaning as above.

All these functional spaces are also used when $\Rd$ is replaced by any open set $\mathcal O \subset \Rd$, with  the same meaning as in the whole space.

We use the symbols $D_tf$, $D_i f$ and $D_{ij}f$ to denote respectively the time derivative $\frac{\partial f}{\partial t}$ and the spatial derivatives $\frac{\partial f}{\partial x_i}$ and $\frac{\partial^2f}{\partial x_i\partial x_j}$ for any $i,j=1, \ldots,d$. The gradient of $f$ is denoted by $
\nabla f$ and the Hessian matrix by $D^2f$.

We denote by $\textrm{Tr}(Q)$ and $\langle x,y\rangle$ the trace of the square matrix
$Q$ and the Euclidean scalar product
of the vectors $x,y\in\Rd$, respectively. The open ball in $\Rd$ centered at $ 0$ with radius $r>0$ and its closure are denoted by  $B_r$ and $\overline{B}_r$, respectively.
For any measurable set $A$ we denote by $\chi_A$ the characteristic function of $A$.

The integral over $\R^d$ of a function $f$ with respect to a measure $\mu$ will be denoted by $\int_{\R^d} f\,d\mu$ or by $\int_{\R^d} f(x)\, \mu(dx)$.

\section{Assumptions and preliminary results}
Let $I$ be either an open right-halfline, or    $I=\R$. Let $\mathcal{A}(t)$ be a family of linear second order differential
operators defined   by
\begin{align}\label{oper}
(\mathcal{A}(t)\zeta)(x)&=\sum_{i,j=1}^d q_{ij}(t,x)D_{ij}\zeta(x)+
\sum_{i=1}^d b_i(t,x)D_i\zeta(x)\\
&= \textrm{Tr}(Q(t,x)D^2\zeta(x))+ \langle b(t,x), \nabla \zeta(x)\rangle,\qquad\;\,t \in I,\,\, x\in \Rd.\nnm
\end{align}

Our standing assumptions on the coefficients of the operators $\A(t)$ are listed   below.

\begin{hyp}\label{base}
\begin{enumerate}[\rm (i)]
\item
$q_{ij}, b_{i}\in C^{\alpha/2,\alpha}_{\rm loc}(I\times \Rd)$ $(i,j=1,\dots,d)$ for
some $\alpha \in (0,1)$;
\item
for every $(t,x)\in I \times \Rd$, the matrix $Q(t,x)=[q_{ij}(t,x)]_{ij}$
is symmetric and uniformly positive definite, i.e.,
\begin{equation*}
\langle Q(t,x)\xi,\xi\rangle\geq \eta(t,x)|\xi|^2,\qquad\;\, (t,x)\in  I \times \Rd, \xi \in \Rd,
\end{equation*}
for some function $\eta:I\times\Rd\to \R$ such that
$$\inf_{(t,x)\in I \times \Rd}\eta(t,x)=\eta_0>0;$$

\item
there exists   $\varphi\in C^2(\R^d)$ with nonnegative values such that
\begin{equation*}
\;\;\;\;\;\qquad\lim_{|x|\to +\infty}\varphi(x)=+\infty \quad\textrm{and}\quad
(\mathcal{A}(t)\varphi)(x)\leq a-c\,\varphi(x),  \quad (t,x)\in I\times \Rd,
\end{equation*}
for some positive constants $a$ and $c$.
\end{enumerate}
\end{hyp}

\noindent
Under Hypothesis \ref{base}  it is possible to define a Markov evolution operator $\{G(t,s):\, t\geq s\in I\}$ in $C_b(\Rd)$  associated to the equation $D_t u=\mathcal{A}(t)u $, see \cite{KunLorLun09Non}. Here we recall its main properties.
For every $f\in C_b(\Rd)$ and any $s \in I$, the function $(t,x)\mapsto (G(t,s)f)(x)$ belongs to  $C_b([s,+\infty)\times \Rd)\cap C^{1,2}((s,+\infty)\times\Rd)$
and it is the unique bounded classical solution of the the Cauchy problem
\begin{eqnarray*}
\left\{
\begin{array}{lll}
D_tu(t,x)=\mathcal{A}(t)u(t,x), &\quad (t,x)\in(s,+\infty)\times\Rd,\\[1mm]
u(s,x)=f(x).
\end{array}
\right.
\end{eqnarray*}
We have
\begin{equation}\label{est_uni}
\|G(t,s)f\|_{\infty}\le \|f\|_{\infty},\qquad\,\,t\in (s,+\infty),\quad f\in C_b(\Rd),
\end{equation}
\noindent
and for any $s\in I$, $t>s$ and every $x \in \Rd$ there exists a unique Borel probability measure $p(t,s,x,\cdot)$ such that
\begin{equation}
(G(t,s)f)(x)=\int_{\Rd}f(y)p(t,s,x,dy),\quad f\in C_b(\Rd).
\label{baglioni-0}
\end{equation}
Moreover for each bounded interval $J \subset I$ and for any $r>0$ the family of the measures $\{p(t,s,x,dy):\, t, \,s\in J, \; t>s, \;x\in  B_r\}$ is tight, i.e., for any $\varepsilon>0$ there exists $\rho>0$ such that $p(t,s,x,\Rd\setminus B_\rho)\le \varepsilon$ for any $t>s\in J$, $x\in  B_r $ (\cite[Lemma 3.5]{KunLorLun09Non}).

\noindent
By \cite[Thm 5.4]{KunLorLun09Non}, there exists an evolution system of measures $\{\mu_t: t\in I\}$   for $G(t,s)$, i.e., for any $t \in I$, $\mu_t$ is a Borel probability measure and
\begin{equation*}
\int_{\Rd}(G(t,s)f)(x)\, \mu_t(dx)=\int_{\Rd}f(x)\,\mu_s(dx),\quad t>s, \;f\in C_b(\R^d).
\end{equation*}
 The Lyapunov function $\varphi$ of Hypothesis \ref{base} belongs to $L^1(\Rd, \mu_t)$ for any $t \in I$ and there exists a positive constant $M_\varphi$ such that
\begin{equation}
\label{nor_fi}
\int_{\Rd}\varphi(x)\mu_t(dx)\le M_\varphi, \qquad\; \, t \in I.
\end{equation}

The measures $\mu_t$ enjoy the following weak  continuity property,

\begin{lemm}
\label{continuita'}
For every $f\in C_b(\R^d)$, the function $t\mapsto \int_{\R^d}f\,d\mu_t$ is continuous in $I$.
\end{lemm}

The proof given in \cite[Cor. 2.3]{LorLunZam10}, that deals with  the time periodic case, works as well in this general case.

The invariance property of the measures  $\{\mu_t\}$, the integral representation formula \eqref{baglioni-0}  and the density of $C_b(\Rd)$ in  $L^p(\Rd, \mu_s)$ for every $s \in I$  (\cite[Lemma 2.5]{AngLorLun}), allow to extend $G(t,s)$ to   $L^p(\R^d, \mu_s) $, see e.g. \cite[p. 2054]{AngLorLun}. Such  extension, still denoted by $G(t,s)$, is a contraction  from $L^p(\Rd, \mu_s)$ to $L^p(\Rd, \mu_t)$, that is,
\begin{equation*}
\|G(t,s)f\|_{L^p(\Rd,\mu_t)}\le \|f\|_{L^p(\Rd,\mu_s)},\qquad\;\, t>s,\,f\in L^p(\Rd,\mu_s).
\end{equation*}
\medskip

Hypothesis \ref{base} is enough to prove continuity properties of mild solutions to linear Cauchy problems, that will be used in the nonlinear case.
For any $[a,b]\subset I$ and any $g \in C_b((a,b)\times \Rd)$, we consider the   function
\begin{equation}
\label{convolution}
v(t,x):=\int_a^t (G(t,r)g(r,\cdot))(x)\,dr,\qquad\;\, t\in [a,b],\, x \in \Rd .
\end{equation}

\begin{lemm}\label{cont}
Let Hypothesis \ref{base} hold. For any $[a,b]\subset I$ and any $g \in C_b([a,b]\times \Rd)$, the function $\{ (t,r):\;a\leq r\leq t \leq b\}\times \R^d\mapsto \R$, $(t,r,x)\mapsto (G(t,r)g(r, \cdot))(x)$ is continuous and bounded.
\end{lemm}
\begin{proof} Boundedness follows immediately from \eqref{est_uni}. We shall prove continuity in the set  $\Lambda_R:= \{ (t,r):\;a\leq r\leq t \leq b\}\times B_R$, for every $R>0$.

Fix $\eps >0$. By the tightness property of the measures $p(t,r,x,dy)$ there exists $\rho>0$ such that
$$\sup \{p(t,s,x,\Rd\setminus B_\rho):\; a\leq s<t\leq b,\;x \in  B_R\} \le \varepsilon .$$
Moreover there exists $\delta_0>0$ such that for $r_1$, $r_2\in [a,b]$ with $|r_1-r_2|\leq \delta_0$ and for every $y\in B_{\rho}$
we have  $|g(r_1,y)-g(r_2,y)|\leq \varepsilon$.

Fix $(t,r,x)$, $(t_0, r_0, x_0)\in \Lambda_R$, such that $t_0 > r_0$. For   $t>r$  (which is not restrictive, since we will let $t\to t_0$ and $r\to r_0$) we have
$$|(G(t,r)g(r,\cdot))(x)-(G(t_0,r_0)g(r_0,\cdot))(x_0)| \le $$
$$\leq |(G(t,r)(g(r,\cdot) -g(r_0, \cdot)  ))(x)| + |(G(t ,r )g(r_0,\cdot))(x )-(G(t_0,r_0)g(r_0,\cdot))(x_0)|$$
Let us estimate the first addendum. We have
$$\begin{array}{l}
|(G(t,r)(g(r,\cdot)- g(r_0,\cdot))(x)|  \le
\\
\\
\leq \ds \int_{B_\rho}|(g(r,y)-g(r_0,y))|p(t,r,x,dy)  + \int_{\Rd\setminus B_\rho}|g(r,y)-g(r_0,y)|p(t,r,x,dy)
\\
\\
\leq  \sup_{y \in B_\rho}|(g(r,y)-g(r_0,y))| p(t,r,x,B_{\rho})+ 2\|g\|_\infty p(t,r,x,\Rd\setminus B_\rho)
\\
\\
\leq \sup_{y \in B_\rho}|(g(r,y)-g(r_0,y))|  + 2\|g\|_\infty p(t,r,x,\Rd\setminus B_\rho),
\end{array}$$
so that, if  $|r-r_0|\leq \delta_0$,
$$ |(G(t,r)g(r,\cdot))(x)-(G(t,r)g(r_0,\cdot))(x)|\leq \varepsilon +2\|g\|_\infty \varepsilon . $$
Moreover, by   \cite[Thm. 3.7]{KunLorLun09Non}, the function $(t,r,x)\mapsto G(t,r)g(r_0, \cdot)(x)$ is continuous in $\{(t,r):\;  t\geq r\in I\}\times \R^d$. Therefore, there exists $\eta >0$ such that $|(G(t ,r )g(r_0,\cdot))(x )-(G(t_0,r_0)g(r_0,\cdot))(x_0)|\leq \eps$ if $|t-t_0| + |r-r_0| + \|x-x_0\|_{\R^d}\leq \eta$. So, $\lim_{(t,r,x)\to (t_0,r_0,x_0)}(G(t,r)g(r,\cdot))(x)= (G(t_0,r_0)g(r_0,\cdot))(x_0)$.

For $t_0 =r_0$, we have $G(t_0, r_0)=I$. If also $t=r$ we have $G(t,r)=I$ and the statement is reduced to the continuity of $g$ at $(r_0, x_0)$. If $t>r$  we argue as above.
\end{proof}

\begin{prop}\label{weak-con}
Let Hypothesis \ref{base} hold. For any $g \in C_b((a,b)\times \Rd)$, the function $v$ belongs to $C_b([a,b]\times\Rd)$.
\end{prop}
\begin{proof}
By estimate \eqref{est_uni},  $v$ is bounded in $[a,b]\times \Rd$. Let us prove that it is continuous. For $a \le t_0 \le t\le b$ and $x$, $x_0 \in \Rd$ we have
$$\begin{array}{ll}
|v(t, x)-v(t_0, x_0)|&\le \ds \left|\int_a^t (G(t,r)g(r,\cdot))(x)\, dr-\int_a^{t_0} (G(t,r)g(r,\cdot))(x)\, dr\right|
\\
\\
&+\ds \left|\int_a^{t_0} (G(t,r)g(r,\cdot))(x)\, dr-\int_a^{t_0} (G(t_0,r)g(r,\cdot))(x_0)\, dr\right|
\\
\\
& \ds \le \int_{t_0}^{t} |(G(t,r)g(r,\cdot))(x)|\, dr
\\
\\
&\ds +\int_a^{t_0} \left|(G(t,r)g(r,\cdot))(x)-(G(t_0,r)g(r,\cdot))(x_0)\right|\, dr.
\end{array}$$
By \eqref{est_uni}, the first integral does not exceed $(t-t_0)\|g\|_{\infty}$. Since for every $r\in(a,b)$ the function $(t,x)\mapsto G(t,r)g(r, \cdot)(x)$ is continuous, then
$$\lim_{(t,x)\to(t_0, x_0)} |(G(t,r)g(r,\cdot))(x)-(G(t_0,r)g(r,\cdot))(x_0)|=0.$$
Moreover, $|(G(t,r)g(r,\cdot))(x)-(G(t_0,r)g(r,\cdot))(x_0)|\leq 2\|g\|_{\infty}$, still by \eqref{est_uni}. By the Dominated Convergence Theorem, the second integral
 vanishes as $(t,x)\to (t_0,x_0)$ with $t\geq t_0 $. Then,  $|v(t,x)-v(t_0,x_0)|$ tends to $0$ as $(t,x)\to {(t_0,x_0)}$ with $t \ge t_0$.
Arguing similarly in the case $t\le t_0$ we get the claim.
\end{proof}

As in Proposition \ref{weak-con}, throughout the paper we shall deal with functions belonging to $C_b(J\times \Rd)$, where $J$ is an interval. Note that if $h\in C_b(J\times \Rd)$, the function $t\mapsto h(t, \cdot)$ is not necessarily continuous with values in $C_b(\R^d)$, as well as $t\mapsto \|h(t, \cdot)\|_{\infty}$. However, the latter function is measurable, as the next lemma shows.

\begin{lemm}\label{misu}
Let $J\subset \R$ be an interval. Then, for any continuous and bounded function $h:J\times \Rd\to \R$, the map $t\mapsto \|h(t, \cdot)\|_\infty$ belongs to $L^{\infty} (J)$.
\end{lemm}
\begin{proof}
First we notice that for any $r>0$, the function $t\mapsto\|h(t,\cdot)\|_{L^\infty(B_r)}$ is continuous in $J$. Indeed, for any $t, t_0 \in J$, we have
\begin{equation}\label{combi}
\left|\|h(t,\cdot)\|_{L^\infty(B_r)}-\|h(t_0, \cdot)\|_{L^\infty(B_r)}\right|\le \|h(t,\cdot)-h(t_0,\cdot)\|_{L^\infty(B_r)}
\end{equation}
and the right hand side of  \eqref{combi} vanishes as $t\to t_0$, by the uniform continuity of $h$ on compact sets.
On the other hand, since
$$\| h(t, \cdot)\|_{\infty} = \sup_{r>0} \|h(t, \cdot)\|_{L^{\infty}(B_r)},\qquad\;\, t \in J,$$
and the supremum of continuous functions is measurable, then $t\mapsto\|h(t,\cdot)\|_\infty$ is measurable.
\end{proof}

Lemma \ref{misu} will be used to apply an $L^{\infty}$ version of the Gronwall Lemma to $h(t):= \|u(t, \cdot)\|_{\infty}$, where $u$ is the mild solution to \eqref{C_P_intro}. In fact, we will use two variants of the Gronwall Lemma.

\begin{lemm}\label{Gronwall}
Let $  a<b \in \R$.
\begin{itemize}
\item[(i)] Let
$w\in L^{\infty}(a,b)$ be a nonnegative  function, and let $h$, $k\ge 0$ be  such that
$$w(t)\leq k +h\int_a^t  w(s)\,ds,\qquad\;\, a.e.\;t \in [a,b]. $$
Then, $w(t)\le e^{h(t-a)}k$ for a.e. $t\in [a,b]$.

\item[(ii)] Let $w\in C([a,b])$ be a nonnegative function, and let $h\leq 0$ be such that
$$w(t)\leq w(s)  +h\int_s^t  w(r)\,dr,\qquad\;\, a\leq s \leq t\leq b. $$
Then, $w(t)\le e^{h(t-a)}w(a)$ for every  $t\in [a,b]$.
\end{itemize}
\end{lemm}

In the following we will need that the local mild solution of the problem \eqref{C_P_intro} with $f \in C_b(\Rd)$, is actually  classical. To this aim we shall use local estimates for the derivatives of $G(t,s)f$.

\begin{prop}
\label{puffpuff}
Let Hypothesis \ref{base} hold. Then for every  $a\in I$, $b>a$,
$R>0$ and $0<\eta \leq 2+\alpha$ there is $C_1= C_1(a,b, R, \eta)>0$ such that for every $f\in C_b(\R^d)$
\begin{equation}
\label{Friedman}
\|(G(t,s)f)_{|\overline{B}_R}\|_{C^{\eta}(\overline{B}_R)} \leq \frac{C_1}{(t-s)^{\eta/2}}\|f\|_{\infty}, \quad a\leq s<t\leq b.
\end{equation}
Moreover, for  $0<\theta \leq 1$, $\theta \leq \eta\leq 2+\alpha $   there is $C_2= C_2(a,b, R, \eta , \theta )>0$ such that for every $f\in C_b(\R^d)$ that is locally $\theta$-H\"older continuous we have
\begin{equation}
\label{noi1}
\|(G(t,s)f)_{|\overline{B}_R}\|_{C^{\eta}(\overline{B}_R)} \leq \frac{C_2}{(t-s)^{(\eta-\theta)/2}}\|f_{|\overline{B}_{R+1}}\|_{C^{\theta}(\overline{B}_{R+1})}, \quad a\leq s<t\leq b.
\end{equation}
\end{prop}
\begin{proof} Estimates \eqref{Friedman} follow from \cite[Thm. 4.6.3]{Fri64Par}, taking $D=\overline{B}_{R+1}\times [s,b]$. To prove \eqref{noi1}
we
use similar estimates for parabolic equations in balls, and a standard localization procedure. We fix $R>0$ and we
denote by $U(t,s)$ the evolution operator associated to the family $\mathcal{A}(t)$, with homogeneous Dirichlet boundary condition, in $C(\overline{B}_{R+1})$. For $0\leq \beta \leq 1$, $\beta \leq \gamma \leq 1+\alpha $ there exists $C_3= C_3(a,b, \gamma, \beta, R)$, such that
\begin{equation}
\label{U1}
 \|U(t,s)\varphi \|_{C^{ \gamma}(\overline{B}_{R+1})} \leq \frac{C_3}{(t-s)^{ (\gamma-\beta)/2}}\|\varphi\|_{C^{\beta}(\overline{B}_{R+1})}, \quad a\leq s<t\leq b,
 \end{equation}
for every $\varphi\in C^{\beta}(\overline{B}_{R+1})$ that vanishes at $\partial B_{R+1}$.  Such estimates should be well known; to be complete we give a proof in the Appendix.

For $s\in I$ set
$$u(t,x): = G(t,s)f(x), \quad t\geq s, \;x\in \R^d.$$
Let $\zeta  \in C^{\infty}_c(\R^d)$ be such that $\zeta \equiv 1 $ in $\overline{B}_R$, $\zeta \equiv 0$ outside $\overline{B}_{R+1}$. The function  $u_1(t,x):= u(t,x)\zeta (x)$ satisfies
 $$\left\{\begin{array}{l}
 D_tu_1= \mathcal{A}(t)u_1 - u \mathcal{A}(t)\zeta  - 2\langle Q(t, \cdot)\nabla_xu, \nabla\zeta \rangle, \quad t>s, \;x\in \overline{B}_{R+1},
 \\
 \\
 u_1(s,x) =  f(x)\zeta (x), \quad x\in \overline{B}_{R+1},
\\
\\
u_1(t,x) =0, \quad t\geq s, \;x\in \partial  B_{R+1},
\end{array}\right. $$
and therefore it is given by
\begin{equation}
\label{eq:u}
u_1(t, \cdot) = U(t,s)(f\zeta) + \int_s^t U(t, r)g_1(r, \cdot)dr, \quad t\geq s,
\end{equation}
where
$$g_1(r,\cdot ) = - u(r,\cdot) \mathcal{A}( r)\zeta  - 2\langle Q(r , \cdot)\nabla_xu(r, \cdot), \nabla\zeta \rangle .$$
$g_1$ is continuous in $(s, +\infty)\times \overline{B}_{R+1}$, it vanishes at $(r,x)$ with $x\in \partial B_{R+1}$,
and for fixed $\sigma\in (0, 1)$ by estimates   \eqref{Friedman} there exists $C_4>0$ independent of $f$  such that
\begin{equation}
\label{eq:g}
\|g_1(r, \cdot)\|_{ C^{\sigma}(\overline{B}_{R+1})}\leq C_4\|f\|_{\infty}(r-s)^{-1/2-\sigma/2}, \quad a\leq s<r\leq b.
\end{equation}
Estimate \eqref{U1} with $\gamma=\eta$, $\beta=\theta$ gives
$$(t-s)^{(\eta-\theta)/2}\|U(t,s)(f\zeta)\|_{C^{\eta}( \overline{B}_{R+1})} \leq C_3 \|f\zeta \|_{C^{\theta}( \overline{B}_{R+1})}\leq C_5\|f\|_{C^{\theta}( \overline{B}_{R+1})} $$
with $C_5$ independent of $f$. Now we fix $\sigma \in (0, 1)\cap (\eta -2, \eta)$;
 estimates \eqref{eq:g} and  \eqref{U1} with $\gamma=\eta$, $\beta =\sigma $  yield
$$\begin{array}{l}
\ds \bigg\| \int_s^t U(t, r)g_1(r, \cdot)dr\bigg\|_{C^{\eta}( \overline{B}_{R+1})}\leq C_1\int_s^t \frac{1}{(t-r)^{(\eta-\sigma)/2} } \|g_1(r, \cdot)\|_{ C^{\sigma}(\overline{B}_{R+1})} dr
\\
\\
\leq  C_6(t-s)^{(1-\eta)/2}\|f\|_{\infty} \leq C_6(b-a)^{(1-\theta)/2} (t-s)^{(\theta-\eta)/2}\|f\|_{\infty},\end{array}$$
with $C_6$ independent of $f$. Recalling \eqref{eq:u}, we obtain
$$\|u_1\|_{C^{\eta}( \overline{B}_{R+1})}\leq C_7 (t-s)^{(\theta-\eta)/2}\|f\|_{C^{\theta}( \overline{B}_{R+1})}, $$
for some $C_7$ independent of $f$, and since $(G(t,s)f)(x)= u_1(t,x)$ for $x\in \overline{B}_R$ the statement follows.
\end{proof}

Taking $\eta=1$, \eqref{Friedman} gives local gradient estimates for $G(t,s)f$. Adding the following
  assumptions to the basic Hypothesis \ref{base}   global gradient estimates are available.

\begin{hyp}\label{smooth}
\begin{enumerate}[\rm (i)]
\item The first order spatial derivatives  of the coefficients $q_{ij}$ and $b_i$ ($i,j=1,\dots,d$)  belong to $C^{\alpha/2,\alpha}_{\rm loc}(I \times \Rd)$;
\item
there exists a continuous function $k:I\to [0,+\infty)$ such that
$$|\nabla_x q_{ij}(t,x)|\le k(t) \eta(t,x), \quad\; \, (t,x) \in I\times \Rd,$$
for any $i,j=1, \dots,d$;
\item there exists a continuous function $m:I\to \R$ such that
\begin{equation*}
\langle\nabla_x b(t,x)\xi,\xi\rangle \le m(t) |\xi|^2,\qquad\;\, \xi \in \Rd,\, (t,x) \in I\times \Rd.
\end{equation*}
\end{enumerate}
\end{hyp}

The  following   gradient estimates were proved in \cite[Thm. 4.11]{KunLorLun09Non}.

\begin{prop}\label{uni-es}
Assume that Hypotheses \ref{base} and \ref{smooth} hold. Then, for any $a  \in I$, $b>a$,  there are $K_1= K_1(a,b)$, $K_2= K_2(a,b)$ such that
\begin{equation}\label{C0-C1}
\|\nabla_x G(t,s)f\|_{\infty}\le \frac{K_1}{\sqrt{t-s}}\|f\|_{\infty},\qquad a\leq s<t\leq b, \;\, f \in C_b(\Rd),
\end{equation}
\begin{equation}\label{C1-C1}
\|\nabla_x G(t,s)f\|_{\infty}\le K_2\|f\|_{C^1_b(\Rd)},\qquad a\leq s<t\leq b ,\;\, f \in C^1_b(\Rd).
\end{equation}
\end{prop}

Global  estimates of the second and third order space derivatives of $G(t,s)f$ are available under stronger assumptions, arguing as in the autonomous case (e.g., \cite{L1,BL1,Lor11Opt}). However, they are not needed here. To prove that mild solutions are in fact classical, local smoothing properties of $G(t,s)$ are enough.

\begin{prop}\label{linear_reg}
Let Hypotheses \ref{base} hold. For $a\in I$, $b>a$,  $g\in C_b([a, b]\times \R^d)$, let
$v$ be the function defined in \eqref{convolution}. Then
\begin{itemize}
\item[(i)]
$v(t, \cdot)\in C^1 (\R^d)$ for $a\leq t\leq b$, and $\sup_{a\leq t\leq b, \, |x|\leq R} |D_jv(t, x)|<+\infty$ for $j=1, \ldots, d$ and for every $R>0$.
\item[(ii)]
If in addition  $g(t,\cdot) $ is  H\"older continuous in every ball, uniformly with respect to $t \in [a,b]$, then $D_{ij} v$, $D_tv$ exist and are continuous in $[a,b]\times\Rd$ for $i$, $j=1, \ldots d$. Moreover,   $D_t v= \A(t)v+g$ in $[a,b]\times\Rd$.
\item[(iii)] If Hypotheses \ref{base} and \ref{smooth} hold, then
$v(t, \cdot)\in C^1_b(\R^d)$ for $a\leq t\leq b$, and $\sup_{a\leq t\leq b, \, x\in \R^d} |D_jv(t, x)|<+\infty$ for $j=1, \ldots, d$.
\end{itemize}
\end{prop}
\begin{proof}
(i) Estimate \eqref{Friedman} with $\eta =1$ allows to differentiate $v$ with respect to $x_j$,   to obtain that $D_jv(t, \cdot)$ is continuous and
$|D_j v(t, x)| \le 2C_1(a,b, R,1) \sqrt{b-a}\|g\|_{\infty}$ for every $t\in [a,b]$, $|x|\leq R$, $j=1, \ldots, d$.

(ii) If $g(t,\cdot) \in C^{\theta}(B_R)$ uniformly with respect to $t \in [a,b]$ for every $R>0$, estimate \eqref{noi1} with $\eta =2$ allows to differentiate continuously $v(t, \cdot)$ twice with respect to the space variables, and to get $|D_{ij}v(t,x)|\leq 2\theta^{-1}C_2(a,b,R, 2, \theta)(b-a)^{\theta/2}\sup_{a\leq r\leq b}\|g(r, \cdot)\|_{C^{\theta}(B_{R+1})}$ for $i$, $j=1, \ldots d$, $x\in \overline{B}_R$.  Continuity in time of the   space derivatives is readily obtained by interpolation. Indeed, taking $\eta \in (2, \min \{ \alpha +2, \theta +2\})$ in  \eqref{noi1} we get $ v(t, \cdot)\in C^{\eta}(\overline{B}_R)$ and $\sup_{a \leq t\leq b}\| v(t, \cdot)\|_{C^{\eta}(\overline{B}_R)} <\infty$. Applying the interpolation inequality
$$\| \varphi\|_{C^2(\overline{B}_R)} \leq C \|\varphi\|_{\infty}^{1-2/\eta} \| \varphi\|_{C^\eta (\overline{B}_R)}^{2/\eta}, \quad \varphi\in C^\eta (\overline{B}_R)$$
to $  v(t, \cdot) - v(s, \cdot)$, we obtain that $t\mapsto v(t, \cdot)$ is continuous (in fact, H\"older continuous) with values in $C^2(\overline{B}_R)$, so that the derivatives $D_iv$, $D_{ij}v$ are continuous in time, uniformly with respect to the space variables in $B_R$. Therefore, $D_iv$, $D_{ij}v$ are continuous in $[a,b]\times\Rd$.

  To conclude, we have to show that $v$ is   differentiable with respect to time  and that $D_t v=\A(t)v+g$ in $[a,b]\times \Rd$.

Let us consider the right derivative. Fix
$t \in [a,b)$, $x \in B_R $ and  $h \in (0, b-t]$. Then,
\begin{align}\label{splitting}
h^{-1}\Big(v(t+h,x)-v(t,x)\Big)=&h^{-1}\int_a^{t}\left(\big(G(t+h,r)-G(t,r)\big)g(r,\cdot)\right)(x) dr\nonumber\\
&\quad + h^{-1}\int_t^{t+h}(G(t+h,r)g(r, \cdot))(x)\, dr\nnm\\
=& \int_a^{t} f_{t,x}(h,r)\, dr+  h^{-1}\int_t^{t+h}(G(t+h,r)g(r, \cdot))(x)dr,
\end{align}
where
$$f_{t,x}(h,r):=h^{-1}\left(\big(G(t+h,r)-G(t,r)\big)g(r,\cdot)\right)(x) .$$
Then
 $$\lim_{h\to 0^+} f_{t,x}(h, r)= \A(t)\big(G(t,r)g(r,\cdot)\big)(x) , \quad  r \in [a,  t). $$
Let us estimate $|f_{t,x}(h,r)|$. Let $K>0$  be such that
$$  |q_{ij}(t,x)|\leq K, \;|b_{i}(t,x)|\leq K, \quad a\leq t\leq b, \;x\in B_R. $$
Then,
using \eqref{noi1} with $\eta =2$ and \eqref{Friedman} with $\eta =1$   we obtain
$$\begin{array}{lll}
&|f_{t,x}(h,r)|  = \ds \bigg|\frac{1}{h} \int_0^h (\A(t+\sigma)G(t+\sigma, r)g(r, \cdot))(x)d\sigma \bigg|
\\
\\
& \leq  \ds \frac{K }{h} \int_0^h \bigg( \sum_{i,j=1}^d |D_{ij}(G(t+\sigma ,r)g(r,\cdot)(x)| +\sum_{i=1}^d|D_i(G(t+\sigma,r)g(r,\cdot)(x)| \bigg)\, dr
\\
\\
& \leq C(\sup_{r\in [a,b]}\|g(r,\cdot)\|_{C^{\theta}( \overline{B}_{R+1})} +\|g\|_{\infty}) (t -r)^{(\theta -1)/2}
\end{array}$$
for some positive constant $C=C(a,b,R,\theta)$.
Hence, by the Dominated Convergence Theorem,
$$\lim_{h \to 0^+}\int_a^{t} f_{t,x}(h,r)dr= \int_a^t \A(t)\big(G(t,r)g(r,\cdot)\big)(x)dr.$$
Concerning the second term in the right hand side of \eqref{splitting}, by Lemma \ref{cont} the function $r\mapsto (G(t+h,r)g(r, \cdot))(x)$ is continuous in $[t,t+h]$.  Therefore there exists $r_h \in (t,t+h)$ such that
$$\frac{1}{h}\int_t^{t+h}(G(t+h,r)g(r,\cdot))(x)\, dr =(G(t+h,r_h)g(r_h,\cdot))(x).$$
Still by Lemma \ref{cont},  the function $(t,r)\to (G(t,r)g(r,\cdot))(x)$ is continuous   for any $x \in \Rd$. Since  $r_h \to t$ as $h \to 0$, we get
$$\lim_{h \to 0^+}\frac{1}{h}\int_t^{t+h}(G(t+h,r)g(r,\cdot))(x)\, dr= g(t,x).$$
Since $R$ is arbitrary,
\begin{equation}\label{studio}
\lim_{h \to 0^+}h^{-1}\Big(v(t+h,x)-v(t,x)\Big)\!=\!\int_a^t \A(t)\big(G(t,r)g(r,\cdot)\big)(x) dr +g(t,x),
\end{equation}
for any $t \in [a,b)$ and $x \in \Rd$.

Let us consider the left derivative. For $t \in (a,b]$, $x \in B_R $ and  $h \in [a-t,0)$ we have
\begin{align*}
h^{-1}\Big(v(t+h,x)-v(t,x)\Big)
&=\!\int_a^{t+h} f_{t,x}(h,r)\, dr\!-\!  h^{-1}\!\int_{t+h}^t\!\!\!\!(G(t,r)g(r, \cdot))(x)\, dr
\end{align*}
and arguing as before we get \eqref{studio}, with $\lim_{h \to 0^-}$ instead of $\lim_{h \to 0^+}$.

Since $D_{ij}v \in C_b([a,b]\times \Rd)$ for $i, j=1, \ldots , d$, we conclude that $D_t v \in C([a,b]\times \Rd)$ and satisfies $D_t v=\mathcal{A}(t)v+g$ in $[a,b]\times \Rd$ as claimed.

(iii) If also Hypothesis \ref{smooth} holds, estimate \eqref{C0-C1} allows to differentiate $v$ with respect to $x_j$ and to obtain that $D_jv(t, \cdot)$ is continuous and satisfies
$\|D_j v(t, \cdot)\|_{\infty} \le 2K_1(a,b) \sqrt{b-a}\|g\|_{\infty}$ for every $t\in [a,b]$, $j=1, \ldots,d$.
\end{proof}

%%%%%%%%%%%%%%%%%%%%%%%%%%%%%%%%%%%%%%%%%%%%%%%%%%%%%%%%%%%%%%%%%%%%%%

\section{Semilinear problems}

%%%%%%%%%%%%%%%%%%%%%%%%%%%%%%%%%%%%%%%%%%%%%%%%%%%%%%%%%%%%%%%%%%%%%%

Fixed $s\in I$, we consider the semilinear parabolic problem
\begin{equation}\label{p_nonlin}
\left\{
\begin{array}{ll}
D_t u(t,x)={\mathcal{A}}(t)u(t,x)+\psi(t,u(t,x)),\quad\quad & (t,x)\in (s,+\infty)\times \Rd,\\[1mm]
u(s,x)= f(x),\quad\quad & x \in \Rd,
\end{array}\right.
\end{equation}
where $\A(\cdot)$ is defined in \eqref{oper}, $f\in X$, $X$ being either $L^p(\Rd,\mu_s)$, $p \in (1,+\infty)$ or $C_b(\Rd)$, and $\psi:I\times \R \to \R$ is a given  function.

\begin{defi} Let $\tau>s \in I$ and $f \in X$. A function $u:[s, \tau]\times \R^d\mapsto \R$ is called
\begin{itemize}
\item[(i)]   \emph{classical solution} of \eqref{p_nonlin} in the interval $[s,\tau]$, if $u\in C^{1,2}((s,\tau]\times \Rd))\cap C([s,\tau]\times \Rd)$ and $u$ satisfies \eqref{p_nonlin};
\item[(ii)]   \emph{mild solution} of \eqref{p_nonlin} in the interval $[s,\tau]$, if for a.e. $x\in \Rd$ and any $t\in (s,\tau]$, the function $r \mapsto \big(G(t,r)\psi(r,u(r,\cdot))\big)(x)$ is integrable in $(s,t)$ and
\begin{equation}\label{defi_mild}
u(t,x)= (G(t,s)u_0)(x)+\int_s^t \big(G(t,r)\psi(r,u(r,\cdot))\big)(x)\, dr,
\end{equation}
for any $t \in [s,\tau]$ and a.e. $x\in \Rd$.
\end{itemize}
\end{defi}

%%%%%%%%%%%%%%%%%%%%%%%%%%%%%%%%%%%%%%%%%%%%%%%%%%%%%%%%%%%%%%%%%%%%%%%

\subsection{Local Existence and Uniqueness of a mild solution}

%%%%%%%%%%%%%%%%%%%%%%%%%%%%%%%%%%%%%%%%%%%%%%%%%%%%%%%%%%%%%%%%%%%%%%%

\subsubsection{The case   $X=L^p(\Rd, \mu_s)$.}
This subsection is devoted to prove existence and uniqueness of a mild solution of   \eqref{p_nonlin} when the initial datum $f$ belongs to $L^p(\Rd,\mu_s)$. To this aim we require that the function $\psi:I \times\R\to \R$ satisfies the following assumptions.

\begin{hyp}\label{hyp1_F}  $\psi(t, \cdot)$ is Lipschitz continuous, uniformly with respect to $t$ in bounded subintervals of $I$, i.e. for any $s \in I$ and $\tau>s$  there exists $L>0$ such that
$$|\psi(t,\xi )-\psi(t,\eta )|\le L|\xi -\eta |,\qquad\;\, t\in [s,\tau],\, \xi, \;\eta  \in \R.$$
Moreover, $t\mapsto \psi(t, 0)$ is in $L^1_{loc}(I)$.
\end{hyp}

In the sequel we will consider the measure
$$\nu(J\times \mathcal O):= \int_J\mu_t(\mathcal O)\,dt,$$
defined on Borel sets $J\subset I$, $\mathcal{O}\subset \Rd$ and canonically extended to the Borel sets of $I\times \R^{d}$.

\begin{thm}\label{loc-exis}
Let   Hypotheses \ref{base}  and  \ref{hyp1_F}  be satisfied.
Then, for any $s \in I$, $\tau >s$ and  $f\in L^p(\Rd,\mu_s)$ there exists a unique mild solution $u_f$ of \eqref{p_nonlin} in $[s,\tau ]$, such that  $u_f$ belongs to $L^{p}([s,\tau ]\times\Rd,\nu)$. There exists $K>0$,  depending only on $s$ and $\tau$,  such that
for every $f$,  $g \in L^p(\Rd,\mu_s)$ we have
\begin{equation}\label{terremoto}
\sup_{s\leq t\leq \tau} \|u_f(t, \cdot) -u_g(t, \cdot) \|_{L^p(\R^d, \mu_t)}\leq   K \|f-g\|_{L^p(\Rd,\mu_s)}.
\end{equation}
\end{thm}
\begin{proof}
We look for a   mild solution in the space
$$Y:= \{ u:[s,\tau]\times\Rd\mapsto \R:\; u \;{\rm is}\;\nu-{\rm measurable}, \;\sup_{s\leq t\leq \tau} \|u(t, \cdot)\|_{L^p(\R^d, \mu_t)} <\infty\}$$
We consider the nonlinear operator $\Gamma$ defined on $Y$ by
\begin{equation}
\label{mild_gamma}
(\Gamma u)(t,x)= G(t,s)f(x)+\int_s^t (G(t,r)\psi(r, u(r, \cdot)))(x)\, dr,\qquad\;\, s\leq t\leq \tau ,
\end{equation}
and we look for a fixed point of $\Gamma$.  To this aim we prove that $\Gamma$  maps $Y$ into itself and it is a contraction provided $Y$ is endowed with the norm
$$\|u\|_Y = \sup_{s\leq t\leq \tau} e^{-\omega t}\|u(t, \cdot)\|_{L^p(\R^d, \mu_t)}  $$
with suitable $\omega >0$.  Note that $Y$ is a Banach space with the norm $\|\cdot \|_Y$, and $Y$ is continuously embedded in $L^p([s,\tau]\times \R^d, \nu)$.

Let $v_1, v_2 \in Y$. Then for $s\leq r\leq \tau$ we have
$$\|\psi(r, v_1(r, \cdot)) -\psi(r, v_2(r, \cdot))\|_{L^p(\R^d, \mu_r)} \leq L \|v_1(r, \cdot) - v_2(r, \cdot)\|_{L^p(\R^d, \mu_r)}, $$
where $L$ is the Lipschitz constant in Hypothesis \ref{hyp1_F}. Since $G(t,r)$ is a contraction from $L^p(\R^d, \mu_r)$ to $L^p(\R^d, \mu_t)$, then
$$ \begin{array}{lll}
e^{-\omega t}\|\Gamma v_1(t, \cdot) -\Gamma v_2(t, \cdot)\|_{L^p(\R^d, \mu_t)} \!\!\!\!& \leq &\!\!\!\!
\ds e^{-\omega t}\!L\int_s^t \|v_1(r, \cdot) - v_2(r, \cdot)\|_{L^p(\R^d, \mu_r)}dr
\\
\\
& \leq &
\omega^{-1}L\|v_1-v_2\|_Y.
\end{array}$$
Therefore, $\Gamma $ is a $1/2$-contraction   if $\omega  \geq 2L$. To prove that the range of $\Gamma$ is contained in $Y$ it is enough to check that $\Gamma(0)\in Y$. This is true since
$$\Gamma(0)(t, x) = (G(t,s)f)(x) + \int_s^t \psi(r, 0)dr, \quad s\leq t\leq \tau, \;x\in \R^d, $$
 so that
$$\|\Gamma (0)(t, \cdot) \|_{L^p(\R^d, \mu_t)} \leq \|f\|_{L^p(\R^d, \mu_s)} +  \bigg|\int_s^t \psi(r, 0)dr\bigg| , $$
which is bounded in $[s, \tau]$ since $\psi(\cdot, 0)\in L^1_{{\rm{loc}}}(I)$.

Let us prove the statement about dependence on the initial datum. For $f$, $g\in L^p(\R^d, \mu_s)$ we have
$$u_f(t, \cdot )-u_g(t, \cdot )= G(t,s)(f-g)+\int_s^t G(t, r)(\psi(r,u_f(r, \cdot )) -\psi(r,u_g(r, \cdot )))dr$$
so that
$$\|u_f -u_g \|_Y \leq  \|f-g\|_{L^p(\R^d, \mu_s)} + \frac{1}{2}\|u_f -u_g \|_Y $$
which implies
$$\|u_f -u_g \|_Y \leq  2\|f-g\|_{L^p(\R^d, \mu_s)}$$
and therefore
$$\sup_{s\leq t\leq \tau} \|u_f(t, \cdot) -u_g(t, \cdot) \|_{L^p(\R^d, \mu_t)}\leq 2e^{\omega \tau  } \|f-g\|_{L^p(\R^d, \mu_s)}$$
which yields \eqref{terremoto}.

Last, we prove uniqueness of the mild solution in $L^p([s, \tau]\times \R^d, \nu)$. We use the same trick as above, namely we endow $L^p([s, \tau]\times \R^d, \nu)$ with the norm
$$\|u\| := \bigg( \int_s^\tau \int_{\R^d} e^{-\omega t }|u(t, x)|^p \mu_t(dx)dt\bigg)^{1/p}, $$
with $\omega  $ large, precisely $\omega >  (\tau -s)^{p-1} L^p$.

If $u_1$, $u_2$ are two mild solutions of \eqref{p_nonlin} belonging to  $L^p([s,\tau]\times \R^d, \nu)$, we have
$$\begin{array}{lll}
&\|u_1-u_2\|^p  =  \ds \int_s^{\tau} e^{-\omega t}\left\| \int_s^t G(t, r)( \psi(r, u_1(r, \cdot)) -  \psi(r, u_2(r, \cdot)))dr\right\|_{L^p(\R^d, \mu_t)}^p dt
\\
\\
&\quad \leq \ds \int_s^{\tau} e^{-\omega t} (t-s)^{p-1} \int_s^t \|G(t, r)( \psi(r, u_1(r, \cdot)) -  \psi(r, u_2(r, \cdot)))\|_{L^p(\R^d, \mu_t)}^p dr\, dt
\\
\\
&\quad \leq \ds \int_s^{\tau} e^{-\omega t} (t-s)^{p-1} \int_s^t \| \psi(r, u_1(r, \cdot)) -  \psi(r, u_2(r, \cdot))\|_{L^p(\R^d, \mu_r)}^p dr\,dt
\\
\\
&\quad \leq \ds \int_s^{\tau} e^{-\omega t} (\tau -s)^{p-1} L^p \int_s^t \|  u_1(r, \cdot) - u_2(r, \cdot)\|_{L^p(\R^d, \mu_r)}^p dr\,dt
\\
\\
& \quad=  \ds  (\tau -s)^{p-1} L^p \int_s^{\tau} e^{-\omega r} \|  u_1(r, \cdot)- u_2(r, \cdot)\|_{L^p(\R^d, \mu_r)}^p \int_r^\tau e^{-\omega( t-r)}dt \,dr
\\
\\
& \quad=    (\tau -s)^{p-1} L^p \omega^{-1} \|u_1-u_2\|^p
\end{array}$$
Since $\omega >  (\tau -s)^{p-1} L^p$, then $\|u_1-u_2\| =0$ so that $u_1=u_2$.
\end{proof}

\subsubsection{The case  $X=C_b(\Rd)$}
Here we prove existence and uniqueness of a local mild solution to
\eqref{p_nonlin} when $f \in C_b(\Rd)$. In this setting, we weaken a part of Hypothesis \ref{hyp1_F} requiring just local Lipschitz continuity of the nonlinearity.

\begin{hyp}\label{hyp_c_b}
The function $\psi$ is continuous and $\psi(t,\cdot)$ is locally Lipschitz continuous, uniformly with respect to $t$ on bounded subintervals of $I$, i.e., for any $s\in I$  and $R>0$ there exists $L_R>0$ such that
\begin{equation}\label{Lip}
|\psi(t,\xi )-\psi(t,\eta )|\le L_R|\xi -\eta |, \qquad\;\, x,y\in [-R,R], t \in [s,s+1] .
\end{equation}
\end{hyp}

\begin{thm}\label{exi_Cb}
Under Hypotheses \ref{base} and \ref{hyp_c_b}, for any $ s \in I$ and any $\overline{f}\in C_b(\Rd)$ there are $r,\delta>0$ such that if $\|f-\overline{f}\|_{\infty}\le r$ then there exists a unique mild solution $u_f \in C_b([s,s+\delta]\times \Rd)$. If  $g \in C_b(\Rd)$ is such that $\|g-\overline{f}\|_{\infty}\le r$, then
\begin{equation}\label{dip_dati}
\|u_f(t,\cdot)-u_g(t,\cdot)\|_{\infty}\le 2 \|f-g\|_{\infty}, \quad\;\, t\in [s,s+\delta].
\end{equation}
\end{thm}

\begin{proof}
Fix $R>0$ such that $R \ge 8\|\overline{f}\|_{\infty}$. If $\|f-\overline{f}\|_{\infty}\le r:=R/8$, then
$$\|G(t,s)f\|_{\infty}\le R/4,\qquad\;\, t>s.$$
We look for a local mild solution in the space
$$Y_R=\{u \in C_b([s,s+\delta]\times \Rd):\,\, \|u\|_{C_b([s,s+\delta]\times \Rd)}\le R\}$$
where $\delta \in (0, 1]$ has to be determined.
We consider the nonlinear operator $\Gamma$ defined on $Y_R$ by  \eqref{mild_gamma}
and we prove that $\Gamma$ is a contraction which maps $Y_R$ into itself, if $\delta$ is small enough.

First of all, for every $v\in Y_R$ the function $(r,x)\mapsto \psi(r, v(r, x))$ is continuous and bounded in $ [s,s+\delta]\times \Rd$; hence by Proposition \ref{weak-con}, $\Gamma v\in C_b([s,s+\delta]\times \Rd)$.

Let $v_1, v_2 \in Y_R$. Then  $\|v_1(t)\|_{\infty}\le R$ and $\|v_2(t)\|_{\infty}\le R$ for any $t \in [s,s+\delta]$ and
\begin{align*}
\|\Gamma(v_1)-\Gamma(v_2) \|_{\infty}
&=\sup_{t\in [s,s+\delta]\atop x\in \Rd}\left|\int_s^t (G(t,r)[\psi(r,v_1(r,\cdot))-\psi(r,v_2(r,\cdot))])(x)\,dr\right|\\
& \le \sup_{t\in [s,s+\delta]}\int_s^{t} \|G(t,r)[\psi(r,v_1(r,\cdot))-\psi(r,v_2(r,\cdot))]\|_{\infty} \,dr\\
& \le \int_s^{s+\delta}\|\psi(r,v_1(r,\cdot))-\psi(r,v_2(r,\cdot))\|_{\infty}\,dr\\
&\le \delta L_R \|v_1-v_2\|_{\infty},
\end{align*}
where $L_R$ denotes the Lipschitz constant in \eqref{Lip}. (Note that the functions $r\mapsto \|G(t,r)[\psi(r,v_1(r,\cdot))-\psi(r,v_2(r,\cdot))]\|_{\infty}$ and $r\mapsto \|\psi(r,v_1(r,\cdot))-\psi(r,v_2(r,\cdot))\|_{\infty}$ are measurable in $(s,t)$ and in $(s, s+\delta)$ respectively, by Lemma \ref{misu}).
Then, choosing $\delta \le \delta_0:= \min\{1, (2 L_R)^{-1}\}$, we obtain that $\Gamma$ is a $\displaystyle{\frac{1}{2}}$-contraction. \\

Let now $v \in Y_R$. If $\delta \le \delta_0$ we have
\begin{align*}
\|\Gamma(v)\|_{\infty}&\le\|\Gamma(v)-\Gamma(0)\|_{\infty}+\|\Gamma(0)\|_{\infty}\\
& \le \frac{R}{2}+\|G(\cdot,s)f\|_{\infty}+ \delta \sup_{r \in [s,T]}|\psi(r,0)|\\
& \le \frac{R}{2}+ \frac{R}{4}+\delta \sup_{r \in [s,T]}|\psi(r,0)|.
\end{align*}
Thus, if $\delta \le \delta_R:=\min\{\delta_0,\delta_1\}$ where $\delta_1=1$ if $\psi(r,0)=0$ for every $r\in[s, s+1]$, $\delta_1:=(\sup_{r \in [s,s+1]}|\psi(r,0)|)^{-1} R/4$ otherwise,
  $\Gamma$ maps $Y_R$ into itself so that it has a unique fixed point in $Y_R$, that is a mild solution of \eqref{p_nonlin}.

To get uniqueness of the mild solution in $C_b([s,s+\delta]\times \Rd)$ we argue by contradiction. Let us assume that $u_1,u_2\in C_b([s,s+\delta]\times \Rd)$ be two mild solutions of \eqref{p_nonlin} and set $R'=\max\{\|u_1 \|_{\infty},\|u_2 \|_{\infty}\}$. For any $t \in [s,s+\delta]$, recalling that
the functions $r\mapsto \|\psi(r,u_1(r,\cdot))-\psi(r,u_2(r,\cdot))\|_\infty$ and $r\mapsto \|u_1(r, \cdot)-u_2(r, \cdot)\|_{\infty}$ are measurable in $(s,t)$ by Lemma \ref{misu},
we have
\begin{align*}
\|u_1(t)-u_2(t)\|_{\infty}&= \left\|\int_s^t \Big(G(t,r)\big(\psi(r,u_1(r,\cdot))-\psi(r,u_2(r,\cdot))\big)\Big)(x)\,dr\right\|_\infty\\
& \le \int_s^t \|\psi(r,u_1(r,\cdot))-\psi(r,u_2(r,\cdot))\|_\infty\,dr\\
&\le L_{R'}\int_s^t\|u_1(r, \cdot)-u_2(r, \cdot)\|_{\infty}\, dr.
\end{align*}
Since $t \mapsto \|u_1(t,\cdot)-u_2(t,\cdot)\|_\infty $ belongs to  $L^{\infty}((s,s+\delta))$, we can apply the  Gronwall Lemma \ref{Gronwall}(i), to deduce that $u_1(t,x)=u_2(t,x)$ for a.e. $t \in [s,s+\delta]$ and for every $x\in\R^d$. Since $u_1$ and $u_2$ are continuous, then $u_1(t,x)=u_2(t,x)$ for every $t \in [s,s+\delta]$,    $x\in\R^d$.

To conclude we prove \eqref{dip_dati}. Let $f$, $g\in B(\overline{f}, r)\subset C_b(\R^d)$. Then $u_f$ and $u_g$ belong to $Y_R$ and
$$u_f(t)-u_g(t)= G(t,s)(f-g)+(\Gamma u_f)(t)-(\Gamma u_g)(t),\quad\;\, t\in [s,s+\delta].$$
Since $\Gamma$ is a $\frac{1}{2}$-contraction in $Y_R$, then
\begin{align*}
\|u_f-u_g\|_{C_b([s,s+\delta]\times \Rd)}\le  2\|G(\cdot,s)(f-g)\|_{C_b([s,s+\delta]\times \Rd)}\le 2\|f-g\|_{\infty}.
\end{align*}
\end{proof}

%%%%%%%%%%%%%%%%%%%%%%%%%%%%%%%%%%%%%%%%%%%%%%%%%%%%%%%%%%%%%%%%%%

\section{Regularity and Global existence}\label{sec_reg}

%%%%%%%%%%%%%%%%%%%%%%%%%%%%%%%%%%%%%%%%%%%%%%%%%%%%%%%%%%%%%%%%%%

This section is devoted to  the regularity of the mild solution given by Theorem \ref{exi_Cb} and to its existence in large. Further regularity properties will be proved  under Hypothesis \ref{smooth}.
First, we show that for every $f\in C_b(\R^d)$ the local mild solution  of problem \eqref{p_nonlin} is actually a classical solution.

\begin{thm}\label{classical}
Assume that Hypotheses \ref{base}  and \ref{hyp_c_b} are satisfied. Then for any $f\in C_b(\R^d)$ the mild solution $u_f \in C_b([s,\tau]\times \Rd)$ of   problem \eqref{p_nonlin}  is a  classical solution. If in addition  Hypothesis \ref{smooth} holds, then
\begin{equation}
\label{uno}
\sup_{s<t\leq \tau, \,x\in\R^d}(t-s)^{1/2}|\nabla_xu_f(t,x)|<\infty , \quad {\it if}\;f\in C_b(\R^d),
\end{equation}
and
\begin{equation}
\label{due}
\sup_{s<t\leq \tau, \,x\in\R^d} |\nabla_xu_f(t,x)|<\infty , \quad {\it if}\; f\in C^1_b(\R^d).
\end{equation}
\end{thm}
\begin{proof}
We split the proof in two steps. In the first step we assume that $f \in C^1_b(\Rd)$. In the second step we complete the proof.

{\em Step 1.} Let $f \in C^1_b(\Rd)$. For any  $t \in (s,\tau]$ and $x \in \Rd$ we set
$g(t,x):=\psi(t,u_f(t,x))$ and we define $v$ as in \eqref{convolution}.
Thus, $u_f(t,x)=(G(t,s)f)(x)+v(t,x)$ for any $t \in [s, \tau]$ and $x \in \Rd$.

Let us notice that, since $u_f$ belongs to $C_b([s,\tau]\times \Rd)$ and $\psi $ satisfies Hypothesis \ref{hyp_c_b}, then  $g \in C_b([s,\tau]\times \Rd)$, hence Proposition \ref{linear_reg}  yields that $ v(t, \cdot)\in C^1 (\R^d)$ for every $t\in [s, \tau ]$, and $\sup_{s\leq t\leq \tau}\|\, |\nabla_xv(t, \cdot)|\,\|_{L^{\infty}(B_R)} < +\infty$ for every $R>0$.  If also Hypothesis \ref{smooth} holds, still Proposition \ref{linear_reg}  yields that $ v(t, \cdot)\in C^1_b (\R^d)$ for every $t\in [s, \tau ]$, and $\sup_{s\leq t\leq \tau}|\nabla_xv(t, \cdot)|_{\infty} < +\infty$.

By \eqref{Friedman} with $\eta =1$ (if only Hypothesis \ref{base} holds) and by \eqref{C1-C1} (if also Hypothesis \ref{smooth} holds)
$ G(t,s)f$ enjoys the same properties, and so does $u_f$. Therefore, \eqref{due} holds if both Hypotheses \ref{base} and \ref{smooth} hold, and it is replaced by
$$\sup_{s\leq t\leq \tau}\|\, |\nabla_xu_f(t, \cdot)|_{L^{\infty}(B_R)} < +\infty, \quad \forall R>0$$
if only Hypothesis \ref{base} holds. In both cases,
 $g(t,\cdot)$ is Lipschitz continuous (hence, $\theta$-H\"older continuous) in each ball $B_R$ uniformly with respect to $t \in [s,\tau]$. In fact,   for any $x,y \in B_R$,
\begin{align*}
|g(t,x)-g(t,y)|&=|\psi(t,u_f(t,x))-\psi(t,u_f(t,y))|\\
&\le L  |u_f(t,x)-u_f(t,y)|\\
& \le L  \sup_{t \in [s,\tau]}\|\,|\nabla_x u_f(t,\cdot)|\, \|_{L^\infty(B_R)}|x-y|,
\end{align*}
where
$$L= \sup\{ |\psi(t,\xi)-\psi(t,\eta)|/|\xi-\eta|: \;s\leq t\leq \tau, \;|\xi|, \;|\eta|\leq \|u_f\|_{\infty}, \;\xi\neq \eta\}$$
is finite by Hypothesis \ref{hyp_c_b}.
Thus, Proposition \ref{linear_reg}  yields that   $D_{ij} v$ belongs to $C ([s,s+\delta]\times \Rd)$ for $i, j=1, \ldots, d$ and that $v_t=\mathcal{A}(t)v+\psi (t,u_f)$ in $[s,s+\delta]\times \Rd$. Consequently, $u_f$ has the same regularity of $v$ and it is a classical solution of   \eqref{p_nonlin} in $[s,s+\delta]$.

{\em Step 2.} Now, let $f \in C_b(\Rd)$. As before, $ v(t, \cdot)\in C^1 (\R^d)$ for every $t\in [s, \tau ]$, and $\sup_{s\leq t\leq \tau}|\nabla_x v(t, \cdot)|_{L^{\infty}(B_R)} < +\infty$ for every $R>0$.  By \eqref{Friedman}, $ G(t,s)f$ enjoys the same properties, and \eqref{uno} follows.

Fix $\varepsilon \in (0,\tau -s)$.
Since  the mild bounded solution in $[s+\varepsilon,  \tau]$ is unique, then
\begin{equation*}
u_f(t,x)= (G(t,s+\varepsilon)u_f(s+\varepsilon,\cdot))(x)+ \int_{s+\varepsilon}^t (G(t,r)\psi(r, u_f(r, \cdot)))(x)\, dr,\
\end{equation*}
for $s+\varepsilon \leq t\leq \tau$, $x\in \R^d$. Since $u_f(s+\varepsilon, \cdot) \in C^{1}_b(\Rd)$, by step 1 applied in the interval $[s+\varepsilon,\tau]$ the restriction of $u_f$ to the interval $[s+\varepsilon,\tau]$ belongs to $ C^{1,2}((s+\varepsilon,\tau]\times \Rd)$ and it is a bounded classical solution of   problem \eqref{p_nonlin} in $[s+\varepsilon,\tau]$. The claim follows.
\end{proof}

\begin{rmk}
Let  hypotheses \ref{base}, \ref{hyp_c_b}, \ref{hyp1_F}  be satisfied. Theorem \ref{classical} and estimate \eqref{terremoto} imply that for every $f\in L^p(\R^d, \mu_s)$ the mild solution of to \eqref{p_nonlin} given by Theorem \ref{loc-exis} is a strong solution, in the sense that $u_f $ is the limit of a sequence of classical solutions $(u_{f_n})$ in $L^p([s, \tau]\times \R^d, \nu)$. It is sufficient to approach $f$ in $L^p(\R^d, \mu_s)$ by a sequence of functions $f_n\in C_b(\R^d)$. More precisely, estimate  \eqref{terremoto}  implies that
$$\lim_{n\to \infty} \sup_{s\leq t\leq \tau} \|u_f(t, \cdot) - u_{f_n}(t, \cdot)\|_{L^p(\R^d, \mu_t)} =0.$$
\end{rmk}

Let $f\in C_b(\Rd)$. The maximal interval of existence of a mild solution to \eqref{p_nonlin} is
$$I(f):=\cup\{[s, s+a]:\, a>0,\, \eqref{p_nonlin}\,\,{\rm has\,\, a\,\,unique\,\, mild\,\, solution\,\,} u_a \in C_b([s,s+a]\times \Rd) \}$$
and the maximally defined solution $u_f:I(f)\times \R^d\to \R$ to \eqref{p_nonlin} is  defined by
$$u_f(t,x) = u_a(t,x), \quad t\in I(f), \;0\leq t\leq a, \;x\in \R^d. $$
Moreover we set
$$\tau_f:= \sup I(f).$$
Thanks to Proposition \ref{weak-con}, the standard procedure to show that either $I(f) = [s, +\infty)$ or  $\|u_f(t, \cdot )\|_{\infty}$ blows up as $t\to \tau_f$ works as well in our situation.
For the sake of completeness we write down a proof.

\begin{lemm}\label{unboun}
Assume that Hypotheses \ref{base} and \ref{hyp_c_b} are satisfied, and let  $f \in C_b(\Rd)$. If $\tau_f<+\infty$, then  $\lim_{t\to \tau_f }\|u_f(t, \cdot)\|_\infty = +\infty$.
\end{lemm}
\begin{proof}
Assume by contradiction that $\|u_f(t, \cdot)\|_\infty$ is bounded. Then the function $(t,x) \mapsto \psi(t, u_f(t,x))$ belongs to $C_b([s,\tau_f)\times \Rd)$, indeed it is continuous and
\begin{align}\label{pianta}
\|\psi(t,u_f(t, \cdot))\|_\infty&\le \|\psi(t,u_f(t,\cdot))-\psi(t,0)\|_\infty+ |\psi(t,0)|\nnm\\
&\le L \|u_f(t,\cdot)\|_\infty+|\psi(t,0)|.
\end{align}
and the right-hand side of \eqref{pianta} is bounded in $I(f)$.
Using Lemma \ref{weak-con} we  extend the mild solution $u_f$ by continuity at $t=\tau_f$. By Theorem  \ref{exi_Cb} there exists $\delta>0$ such that the problem
\begin{equation*}
\left\{
\begin{array}{ll}
D_t v(t)={\mathcal{A}}(t)v(t)+\psi(t,v(t)),\quad\quad & t\in (\tau_f,+\infty),\\[1mm]
v(\tau_f)=u(\tau_f),
\end{array}\right.
\end{equation*}
has a unique mild solution $v \in C_b([\tau_f,\tau_f+\delta]\times \Rd)$. The  function
\begin{equation*}
w(t,x)=\left\{
\begin{array}{ll}
u_f(t,x),\quad\quad & t\in [s,\tau_f),\;x\in \R^d,\\[1mm]
v(t,x),\quad\quad & t\in [\tau_f,\tau_f+\delta], \;x\in \R^d,
\end{array}\right.
\end{equation*}
is a mild solution of \eqref{p_nonlin} belonging to $C_b([s,\tau_f+\delta]\times \Rd)$, contradicting the definition of $\tau_f$. Hence the claim is proved.
\end{proof}

\begin{prop}\label{global}
Assume that Hypotheses \ref{base} and \ref{hyp_c_b} are satisfied,  and that for every $s\in I$, $\tau >s$ there exists a positive constant $h$ such that
\begin{equation}\label{suff_glob}
|\psi(t,\xi)|\le h(1+|\xi|), \qquad\,\, t \in [s,\tau ],\, \xi \in \R.
\end{equation}
Then $I(f) = [s, +\infty)$ for any $f \in C_b(\Rd)$.
\end{prop}
\begin{proof}
Assume by contradiction that $\tau_f<+\infty$, and take $\tau = \tau_f$ in \eqref{suff_glob}. By Lemma \ref{misu} the function $r\mapsto \|u_f(r,\cdot)\|_{\infty}$ is measurable in $I(f)$, and  using \eqref{suff_glob} we get
$$
\|u_f(t,\cdot)\|_\infty \le \|f\|_\infty+ h(\tau_f-s)+  h\int_s^t\|u_f(r,\cdot)\|_{\infty}\,dr,\quad t \in I(f).
$$
Hence, Lemmas \ref{misu} and  \ref{Gronwall} yield
$$\|u_f(t,\cdot)\|_{\infty}\le c (\|f\|_\infty+(\tau_f-s)),\qquad\;\, t \in [s,\tau_f),$$
for some positive constant $c$ independent of $f$. So, $u_f$ is bounded, contradicting Lemma \ref{unboun}.
\end{proof}

As in the case of bounded coefficients, condition \eqref{suff_glob} may be considerably weakened (namely, replaced by a one-sided condition)  if the mild solution is classical. The key assumption here is Hypothesis \ref{base}(iii), that allows to extend the usual maximum principle arguments (e.g., \cite[Thm. 2.9, Ch. 1]{LadSolUra68Lin}) to our situation.

\begin{thm}
\label{Exlarge}
Let Hypotheses \ref{base}  and \ref{hyp_c_b}
 hold. Moreover, assume that for every $s  \in I$, $\tau >s$ there exists $k>0$ such that
\begin{equation}\label{alge}
\xi \psi(t,\xi )\le k(1+\xi^2),\qquad\;\, t\in [s,\tau],\,\xi \in \R .
\end{equation}
Then $I(f) = [s,   +\infty)$ for every
$f \in C_b(\Rd)$.
\end{thm}
\begin{proof}
Assume that $\tau_f$ is finite, and let $k$ be the constant in \eqref{alge} with $\tau = \tau_f$.
In view of Lemma \ref{unboun}, it suffices to prove that $t \mapsto \|u_f(t)\|_\infty$ is bounded in $I(f)$. First we prove that $u_f$ is bounded from above. To this aim, we fix $b\in (0, \tau_f-s)$ and $\lambda >k$, and we set
\begin{equation}
\label{vn}
v_n(t,x)= e^{-\lambda(t-s)}u_f(t,x) -\frac{\varphi(x)}{n},\quad\;\, t\in [s,s+b],\, x \in \Rd.
\end{equation}
Then,
\begin{equation}
\label{multa}
\begin{array}{l}
D_t v_n(t,x) -(\A(t) v_n)(t, x) =  \ds -\lambda\bigg(v_n (t,x)+ \frac{\varphi(x)}{n}\bigg)
\\
\\
+\ds  \frac{(\A(t) \varphi)(x)}{n} + e^{-\lambda(t-s)}\psi(t, (v_n(t, x) +\varphi(x)/n)e^{\lambda(t-s)})
\end{array}
\end{equation}
for $s< t\leq s+b$, $x\in \R^d$. Since $u_f$ is bounded and $\lim_{|x|\to \infty}\varphi(x) = +\infty$, then $v_n$ has a maximum point $(t_n, x_n)$. If
$v_n(t_n, x_n) \leq 0$ for every $n$, then $ u_f(t,x) \leq 0$ for every $(t,x)\in  [s,s+b]\times \R^d$. Assume that  $v_n(t_n, x_n) > 0$ for some $n$. If $t_n=s$, then $v_n(t_n, x_n) \leq \sup f$. If $t_n>s$, $D_t v_n(t_n,x_n) -(\A(t_n) v_n)(t_n, x_n)\geq 0$ so that, multiplying both sides of \eqref{multa} at $(t_n,x_n)$ by
 $v_n(t_n,x_n)+ \varphi(x_n)/n>0$ and using Hypohesis \ref{base}(iii)  and \eqref{alge},
 we get
$$\begin{array}{lll}
0 & \leq & \ds  -\lambda\bigg(v_n (t_n,x_n)+ \frac{\varphi(x_n)}{n}\bigg)^2 + \frac{a-c \varphi(x_n)}{n}\bigg(v_n (t_n,x_n)+ \frac{\varphi(x_n)}{n}\bigg)
\\
\\
& & \ds + k\bigg( 1+ \bigg(v_n (t_n,x_n)+ \frac{\varphi(x_n)}{n}\bigg)^2\bigg)
\end{array}$$
which implies
$$(\lambda -k) \bigg(v_n (t_n,x_n)+ \frac{\varphi(x_n)}{n}\bigg)^2 -\frac{a}{n} \bigg(v_n (t_n,x_n)+ \frac{\varphi(x_n)}{n}\bigg) \leq k.$$
Therefore, $v_n (t_n,x_n)+ \varphi(x_n)/n \leq \overline{\xi}_n$, where $\overline{\xi}_n$ is the positive solution to $(\lambda -k)\xi^2 -a \xi /n=k$.
So, we get
$$v_n(t,x) \leq \max\{ 0, \sup f, \overline{\xi}_n\}, \quad s\leq t\leq t+b, \;x\in \R^d, $$
and letting  $n\to \infty$,
$$ u_f(t,x)\leq e^{\lambda (\tau_f-s)} \max\{ 0, \sup f, \sqrt{k/(\lambda -k)}\}, \quad s\leq t\leq t+b, \;x\in \R^d, $$
which is an upper bound for $ u_f$, independent of $b$. The same procedure, with $v_n$ replaced  by $  e^{-\lambda(t-s)}u_f(t,x) +  \varphi(x)/n$, gives a similar lower bound. Since $b$ is arbitrary, we get $\|u_f \|_{C_b([s,\tau(f))\times \R^d)}<+\infty$, and the claim is so proved by contradiction.
\end{proof}

 %%%%%%%%%%%%%%%%%%%%%%%%%%%%%%%%%%%%%%%%%%%%%%%%%%%%%%%%%%%%%%%%%%
\section{Stability of the null solution}
%%%%%%%%%%%%%%%%%%%%%%%%%%%%%%%%%%%%%%%%%%%%%%%%%%%%%%%%%%%%%%%%%%%%

In this section we  assume that $\psi(t, 0)=0$ for every $t\in I$, and we study the stability of the null solution to
\begin{equation}\label{eq}
  D_t u(t,x)=\mathcal{A}(t)u(t,x)+\psi(t, u(t,x)),\qquad\;\, t \in (s,+\infty), x \in \Rd,
\end{equation}
 in the space $C_b(\Rd)$ and in the  spaces   $L^p(\R^d, \mu_t)$.

The definition of stability, instability and asymptotic stability in $C_b(\R^d)$ is the usual one; the definition of stability in our time dependent $L^p$ spaces is less standard.

\begin{defi}\label{stabili_b}
Let $I(f)=[s,+\infty)$. We say that the trivial solution $u(t)\equiv 0$ of the equation \eqref{eq} is
\begin{enumerate}[\rm (i)]
\item
{\emph{stable in}} $C_b(\Rd)$ if for any $\varepsilon>0$ and $s\in I$ there exists $\delta>0$ such that if $f\in C_b(\Rd)$ satisfies $\|f\|_\infty\leq \delta$ then $\|u_f(t,\cdot)\|_\infty\leq \varepsilon$ for any $t\ge s$;
\item
{\emph{stable in }}$L^p(\Rd,\mu_t)$ if for any $\varepsilon>0$ and $s\in I$ there exists $\delta>0$ such that if $f\in L^p(\Rd,\mu_s)$ satisfies $\|f\|_{L^p(\Rd,\mu_s)}\leq \delta$ then $\|u_f(t,\cdot)\|_{L^p(\Rd,\mu_t)}\leq \varepsilon$ for any $t\ge s$;

\item{\emph{asymptotically stable in}} $C_b(\Rd)$ if (i) holds and there exists $\delta>0$ such that for any $f\in C_b(\Rd)$ with $\|f\|_\infty\leq \delta$ the function $u_f(t,\cdot)$ converges to $0$ uniformly in $\Rd$ as $t \to +\infty$;
    \item{\emph{asymptotically stable in}} $L^p(\Rd,\mu_t)$ if (ii) holds and there exists $\delta>0$ such that if $f\in L^p(\Rd,\mu_s)$ satisfies $\|f\|_{L^p(\Rd,\mu_s)}\leq \delta$ then $\|u_f(t,\cdot)\|_{L^p(\Rd,\mu_t)}$ converges to $0$ as $t \to +\infty$;
\
\item{\emph{unstable in }} $C_b(\Rd)$ $($resp. $L^p(\Rd,\mu_t))$ if it is not stable in $C_b(\Rd)$ $($resp. $L^p(\Rd, \mu_t))$.
\end{enumerate}
\end{defi}

\begin{rmk}  It is clear that each sufficient condition which guarantees that the trivial solution of the ordinary differential equation $u'=\psi(t, u)$ is unstable, also guarantees that the trivial solution of the partial differential equation $D_t u=\A(t)u+\psi(t, u)$ is unstable.
\end{rmk}

\medskip
In next Theorems \ref{asystable} and \ref{martina} we shall give sufficient conditions for the stability of the trivial solution $u\equiv 0$  in $C_b(\Rd)$ and in $L^p(\Rd, \mu_t)$ respectively.
To this aim we will consider the following assumptions.

\begin{hyp}\label{reg}
\begin{enumerate}[\rm(i)]
\item the function $\psi(t, \cdot)$ is continuously differentiable at $0$ and the function $\partial_y\psi(t,0)=\frac{\partial}{\partial y}\psi(t,y)_{|y=0}$ belongs to $C^{\alpha/2}_{\rm loc}(I)$;
\item $ \sup_{t\in I} \partial_{\xi}\psi(t, 0) =: -\omega_0<0$.
    \end{enumerate}
\end{hyp}

In view of Hypotheses \ref{reg}, we   write   equation \eqref{eq} as
\begin{equation*}
D_tu(t,x)= \mathcal{B}(t)u(t,x)+\Phi(t,u(t,x)),\qquad\;\,t \in (s,+\infty),\, x \in \Rd,
\end{equation*}
where
\begin{equation*}
\mathcal{B}(t)v(x) :=\mathcal{A}(t)v(x)+\partial_y\psi(t,0))v(x), \quad t\in I, \;x\in \R^d,
\end{equation*}
\medskip
and
\begin{equation}\label{fit}
\Phi(t,\xi)= \psi(t,\xi)-\partial_y\psi(t,0)\xi,\quad t\in I, \;\xi \in \R.
\end{equation}

We denote by $G_{\mathcal B}(t,s)$ the evolution operator associated to the family of operators $\mathcal{B}(t)$ in $C_b(\Rd)$. It is easy to show that $G_{\mathcal B}(t,s)$ can be written in terms of   $G(t,s)$ as
\begin{equation}\label{explicit}
G_{\mathcal B}(t,s)f= \exp\Big(\int_s^t \partial_y\psi(\sigma,0)\, d \sigma\Big)G(t,s)f, \qquad\;\, f \in C_b(\Rd), s\in I, \;t\geq s.
\end{equation}
Estimate \eqref{est_uni} yields
\begin{equation*}
\|G_{\mathcal B}(t,s)f\|_{\infty}\le \exp\Big(\int_s^t \partial_y\psi(\sigma,0)\, d \sigma\Big)\|f\|_{\infty},
\end{equation*}
for any $ f\in C_b(\Rd)$ and $t\ge s\in I$.\\

As usual, it will be useful to consider   exponentially weighted $C_b$ spaces. For any $\omega \in \R$ and $s\in I$ we define $C_{\omega}([s,+\infty)\times \Rd)$ as the set of the continuous   functions $v:[s,+\infty) \to C_b(\Rd)$ such that
\begin{equation*}
\|v\|_{C_{\omega}([s,+\infty)\times \Rd)}:=\sup_{t\in[s,+\infty)}e^{\omega (t-s)}\|v(t,\cdot)\|_{\infty}<+\infty.
\end{equation*}
Clearly, $C_{\omega}([s,+\infty)\times \Rd)\subset C_b([s,+\infty)\times \Rd)$ if $\omega \ge 0$.

\begin{prop}\label{prop_lin}
Let Hypotheses \ref{base}  and \ref{reg} hold. Fix $\omega \in [0, \omega_0)$. For any
 $f \in C_b(\Rd)$  and $g \in C_{\omega}([s,+\infty)\times \Rd)$, let $z$ be the unique mild solution of the problem
\begin{equation*}
\left\{
\begin{array}{ll}
D_t z(t,x)=\mathcal{B}(t)z(t,x)+g(t,x),\quad\quad & t\in (s,+\infty),\,x\in \Rd\\[1mm]
z(s,x)= f(x),\quad\quad & x\in \Rd,
\end{array}
\right.
\end{equation*}
Then $z$ belongs to $C_{\omega}([s,+\infty)\times \Rd)$ and
\begin{equation}
\label{omega}
\|z\|_{C_{\omega}([s,+\infty)\times \Rd)}\le  \|f\|_{\infty}+ \frac{1}{\omega_0-\omega} \|g\|_{C_{\omega}([s,+\infty)\times \Rd)}.
\end{equation}
\end{prop}
\begin{proof}
The function $v(t,x):=e^{\omega(t-s)}z(t,x)$ is the unique mild solution of the problem
\begin{equation*}
\left\{
\begin{array}{ll}
D_t v(t,x)=(\mathcal{B}(t)+\omega)v(t,x)+e^{\omega(t-s)}g(t,x),\quad\quad & (t,x)\in (s,+\infty)\times \Rd,\\[1mm]
v(s,x)= f(x),\quad\quad & x \in \Rd,
\end{array}
\right.
\end{equation*}
so that it  is given by the variation of constants formula
$$v(t,x)= e^{\omega(t-s)}(G_{\mathcal B}(t,s)f)(x)+\int_s^t e^{\omega(t-r)}(G_{\mathcal B}(t,r)g_\omega(r, \cdot))(x)\,dr\quad\;t>s,\, x \in \Rd,$$
where   $g_\omega(r,x)=e^{\omega(r-s)}g(r,x)$ for any $s< r < t$ and $x \in \Rd$.
Since
$$e^{\omega(t-s)}\|G_{\mathcal B}(t,s)\|_{\mathcal{L}(C_b(\Rd))}\le e^{(\omega+h_0)(t-s)}\le 1, \quad s\in I, \;t\geq s, $$
then
$$\begin{array}{ll}
\|v(t, \cdot)\|_{\infty}& \ds \le \|f\|_{\infty}+ \int_s^t e^{(\omega+h_0)(t-r)}\|g_\omega(r, \cdot)\|_{\infty}\, dr
\\
\\
& \ds \le \|f\|_{\infty}+ \|g\|_{C_{\omega}([s,+\infty)\times \Rd)}\int_s^t e^{(\omega-\omega_0)(t-r)}\, dr
\\
\\
& \ds \le \|f\|_{\infty}+\frac{1}{\omega_0-\omega} \|g\|_{C_{\omega}([s,+\infty)\times \Rd)},
\end{array}$$
for any $t \in [s,+\infty)$. Taking the supremum with respect to $t \in [s,+\infty)$, \eqref{omega} follows.
\end{proof}

Proposition \ref{prop_lin} is used to prove a nonautonomous version of the principle of linearized stability, in the spirit of
\cite{Hen81Geo}.

\medskip

\begin{thm}\label{asystable}
Let Hypotheses \ref{base}, \ref{hyp_c_b}, \ref{reg} hold. Fix $s\in I$ and assume in addition that
the function $\partial_\xi\psi(t, \cdot)$ be continuous in a neighborhood $\mathcal{U}_0$ of $0$ uniformly with respect to $t \geq s$. Then, for any $\omega \in [0,\omega_0)$ there exists $ r_\omega>0$ such that if $f\in C_b(\Rd)$ and $\|f\|_{\infty}\le r_\omega$ then $\tau(f)=+\infty$ and the unique mild solution $u_f$ of   problem \eqref{p_nonlin} satisfies
\begin{equation}\label{aim1}
\|u_f(t,\cdot)\|_{\infty}\le 2e^{-\omega(t-s)}\|f\|_{\infty},\qquad\; t \in [s, +\infty).
\end{equation}
In particular, the trivial solution is asymptotically stable in $C_b(\Rd)$.
\end{thm}
\begin{proof}
First of all we claim that the function
$$K_t(\rho)= \sup\left\{\frac{|\Phi(t,\xi)-\Phi(t,\eta)|}{|\xi-\eta|}:\quad\, \xi,\eta \in [-\rho,\rho] \right\},$$
where $\Phi$ is defined in \eqref{fit}, goes to $0$ as $\rho \to 0^+$ uniformly with respect to $t \in [s,+\infty)$.
Indeed,   $\Phi(t, \cdot)$ is continuously differentiable in a neighborhood of $0$ and
\begin{equation}\label{es}
K_t(\rho) \le \sup_{r \in (s,+\infty)}\sup_{\sigma\in [-\rho,\rho] }\Big(\partial_\xi\psi(r, \sigma )-\partial_\xi\psi(r, 0)\Big),
\end{equation}
for any $t \in [s,+\infty)$, $\rho>0$. The right hand side in \eqref{es} goes to $0$ as $\rho \to 0^+$, and the claim follows.

Now we show that if $\|f\|_{\infty}$ is small enough, the   solution $u_f$ of   \eqref{p_nonlin} is also the unique fixed point of the operator $\Gamma$ defined on
$$Y_\rho=\left\{v\in C_b([s,+\infty)\times \Rd):\;\,\sup_{t\ge s} e^{\omega(t-s)}\|v(t,\cdot)\|_{\infty}\leq \rho\right\},$$
by setting
\begin{equation*}
(\Gamma v)(t)= G_{\mathcal B}(t,s)f+\int_s^t G_{\mathcal B}(t,r)\Phi(r, v(r))\, dr,\qquad\;\, t \ge s,\, v\in Y_\rho.
\end{equation*}
(See the notation after Hypothesis  \ref{reg}).  Lemma \ref{weak-con} and formula \eqref{explicit} imply  that $\Gamma v \in C_{\omega}([s,+\infty)\times \Rd)$. Moreover, if $v \in Y_\rho$, then
\begin{align}\label{fi}
\|\Phi(t, v(t, \cdot))\|_{\infty}&=\|\Phi(t, v(t, \cdot))-\Phi(t,0)\|_{\infty}\nnm\\
&\le K_t(\rho)\|v(t,\cdot)\|_{\infty}\nnm\\
& \le K_t(\rho) e^{-\omega(t-s)}\rho,\quad\;\, t \in (s,+\infty),
\end{align}
so that the function $(t,x)\mapsto g(t,x):=\Phi(t, v(t,x))$ belongs to $C_{\omega}([s,+\infty)\times \Rd)$, and  $\|g\|_{C_\omega([s,+\infty)\times \Rd)}\le \rho K_t(\rho)$. Applying   Proposition \ref{prop_lin} we obtain that $\Gamma v \in C_{\omega}([s,+\infty)\times \Rd)$ and
\begin{equation}\label{da_dim}
\|\Gamma v\|_{C_{\omega}([s,+\infty)\times \Rd)}\le  \|f\|_{\infty}+ \frac{1}{\omega_0-\omega}\|g\|_{C_{\omega}([s,+\infty)\times \Rd)} .
\end{equation}
Choosing $\rho>0$ small enough such that for any $t \in [s,+\infty)$, $K_t(\rho)\le (\omega_0-\omega)/2 $ and $\|f\|_{\infty}\le r_{\omega}:=   \rho/2$, we obtain $\Gamma v \in Y_\rho$.
Moreover, for any $v_1,v_2 \in Y_\rho$ we have
$$(\Gamma v_1)(t)-(\Gamma v_2)(t)=\int_s^t G_{\mathcal B}(t,r)(\Phi(r,v_1(r))-\Phi(r, v_2(r)))\, dr,$$
hence, by \eqref{omega},
$$\|\Gamma v_1-\Gamma v_2\|_{C_{\omega}((s,+\infty)\times \Rd)}\le \frac{1}{\omega_0-\omega} \|\Phi(\cdot,v_1(\cdot))-\Phi(\cdot, v_2(\cdot))\|_{C_{\omega}([s,+\infty)\times \Rd)}.$$
On the other hand,  if $v \in Y_\rho$ then $\sup_{t \in [s,+\infty)}\|v(t,\cdot)\|_{\infty}\le \rho$, so that
\begin{align*}
\|\Phi(t,v_1(t,\cdot))-\Phi(t, v_2(t,\cdot))\|_{\infty}\le K_t(\rho)\|v_1(t,\cdot)-v_2(t,\cdot)\|_{\infty}
\end{align*}
and
\begin{align*}
\|\Gamma v_1-\Gamma v_2\|_{C_{\omega}([s,+\infty)\times \Rd)}&\le  \frac{1}{\omega_0-\omega} \sup_{t \in [s,+\infty)}e^{\omega(t-s)}K_t(\rho)\|v_1(t,\cdot)-v_2(t,\cdot)\|_{\infty} \\
& \le \frac{1}{2}\| v_1- v_2\|_{C_{\omega}([s,+\infty)\times \Rd)}.
\end{align*}
Hence $\Gamma$ is a $\frac{1}{2}$- contraction on $Y_\rho$, and it admits a unique fixed point $\overline{v}\in Y_\rho$ that is a mild solution of \eqref{p_nonlin},  and therefore it coincides with $u_f$. In particular $u_f \in Y_{\rho}$, and   using \eqref{fi} and \eqref{da_dim} we get
\begin{align*}
\|u_f\|_{C_{\omega}([s,+\infty)\times \Rd)}& = \|\Gamma u_f\|_{C_{\omega}([s,+\infty)\times \Rd)}
\\
&\le \|\Gamma u_f-\Gamma(0)\|_{C_{\omega}([s,+\infty)\times \Rd)} + \|\Gamma(0)\|_{C_{\omega}([s,+\infty)\times \Rd)} .
\end{align*}
Since $\Gamma(0)(t,x)= (G_{\mathcal B}(t,s)f)(x)$, and $\|G_{\mathcal B}(\cdot ,s)f\|_{C_{\omega}([s,+\infty)\times \Rd)}\le   \|f\|_{\infty}$,
  \eqref{aim1} follows.
\end{proof}

\begin{rmk}
Looking at the proof of Theorem \ref{asystable}, we see that if $\partial_\xi\psi(t, \cdot)$ is continuous in a neighborhood $\mathcal{U}_0$ of $0$ uniformly with respect to $t \in I$, then $r_\omega$ does not depend on $s$.
\end{rmk}

If we strenghten condition \eqref{alge}, replacing it by
\begin{equation}
\label{psi0}
\xi\,\psi(t,\xi)\le \psi_0\, \xi^2,\qquad\;\, t \in I,\,\xi \in \R,
\end{equation}
we obtain better estimates, that yield a global stability result if  $\psi_0\leq 0$.

\begin{thm}
\label{globalC0}
Let Hypotheses \ref{base} and \ref{hyp_c_b} hold. If there exists $\psi_0\in \R$ such that \eqref{psi0} holds, then for every $s\in I$ and $f\in C_b(\R^d)$ we have
$I(f) = [s, +\infty)$, and
\begin{equation}
\label{expC0}
|u_f(t, x)|\leq e^{\psi_0(t-s)}\|f\|_{\infty}, \quad t\geq s, \;x\in \R^d.
\end{equation}
\end{thm}
\begin{proof}
For every $s\in I$ and $\tau >s$,   \eqref{alge} is satisfied, and therefore, by Theorem \ref{Exlarge}, $I(f) = +\infty$ for every $f\in C_b(\R^d)$.
To obtain estimate \eqref{expC0} we modify the proof of Theorem \ref{Exlarge}. We define $v_n$ by \eqref{vn}
taking now $\lambda >\psi_0$ and $b=+\infty$. Since $u_f$ is bounded and $\lim_{|x|\to \infty}\varphi(x) = +\infty$, then $v_n$ has a maximum point $(t_n, x_n)$. If
$v_n(t_n, x_n) \leq 0$ for every $n$, then $ u_f(t,x) \leq 0$ for every $(t,x)\in  [s,+\infty)\times \R^d$. Assume that  $v_n(t_n, x_n) > 0$ for some $n$. If $t_n=s$, then $v_n(t_n, x_n) \leq \sup f$. If $t_n>s$, $D_t v_n(t_n,x_n) -(\A(t_n) v_n)(t_n, x_n)\geq 0$ so that, multiplying both sides of \eqref{multa} at $(t_n,x_n)$ by
 $v_n(t_n,x_n)+ \varphi(x_n)/n>0$ and using Hypohesis \ref{base}(iii)  and \eqref{psi0},
 we get
$$\begin{array}{lll}
0 & \leq & \ds  -\lambda\bigg(v_n (t_n,x_n)+ \frac{\varphi(x_n)}{n}\bigg)^2 + \frac{a-c \varphi(x_n)}{n}\bigg(v_n (t_n,x_n)+ \frac{\varphi(x_n)}{n}\bigg)
\\
\\
& & \ds +   \psi_0  \bigg(v_n (t_n,x_n)+ \frac{\varphi(x_n)}{n}\bigg)^2
\end{array}$$
which implies
$$(\lambda -\psi_0) \bigg(v_n (t_n,x_n)+ \frac{\varphi(x_n)}{n}\bigg)^2 -\frac{a}{n} \bigg(v_n (t_n,x_n)+ \frac{\varphi(x_n)}{n}\bigg) \leq 0.$$
Therefore, $v_n (t_n,x_n)+ \varphi(x_n)/n \leq a(\lambda-\psi_0)^{-1}$, that yields
$$v_n(t,x) \leq \max\bigg\{ 0, \sup f, \;\frac{a}{n(\lambda-\psi_0)}\bigg\}, \quad t\geq s, \;x\in \R^d. $$
Coming back to $u_f$ we get
$$e^{-\lambda(t-s)}u_f(t,x)-\frac{\varphi(x)}{n} \leq  \max\bigg\{ 0, \;\sup f, \;\frac{a}{n(\lambda-\psi_0)}\bigg\}, \quad t\geq s, \;x\in \R^d. $$
Letting $n\to \infty$, we obtain
$$e^{-\lambda(t-s)}u_f(t,x)  \leq  \max\{ 0, \;\sup f \}, \quad t\geq s, \;x\in \R^d,  $$
and letting $\lambda \to \psi_0$,
$$e^{-\psi_0(t-s)}u_f(t,x)  \leq  \max\{ 0, \;\sup f\}, \quad t\geq s, \;x\in \R^d. $$
Arguing similarly, with $v_n$ defined now  by $  e^{-\lambda(t-s)}u_f(t,x) +  \varphi(x)/n$, we obtain
$$e^{-\psi_0(t-s)}u_f(t,x)  \geq  \min\{ 0, \; \inf f \}, \quad t\geq s, \;x\in \R^d, $$
and \eqref{expC0} follows.
\end{proof}

Condition \eqref{psi0} allows to obtain   global estimates also in the context of our $L^p$ spaces.
We would like to follow the standard method to get $L^p$ estimates of classical solutions for a fixed measure, together with the heuristic formula
\begin{equation}
\label{heuristic}
 D_t \int_{\R^d} g(x)  \mu_t(dx) = -\int_{\R^d}(\A(t)g )(x) \mu_t(dx),
\end{equation}
that would give (with $u=u_f$)
$$\begin{array}{l}
\ds D_t \int_{\R^d} |u(t,x)|^p \mu_t(dx) =   \int_{\R^d} (D_t |u(t,x)|^p- (\A(t)(|u(t,\cdot )|^p))(x))\mu_t(dx)
\\
\\
= \ds \int_{\R^d}p |u(t,x)|^{p-2}(u(t,x)\psi(t, u(t,x)) -  (p-1) \langle Q(t,x)\nabla_xu(t,x), \nabla_xu(t,x)\rangle )\mu_t(dx)
\\
\\
 \leq  \psi_0 p \ds \int_{\R^d} |u(t,x)|^p \mu_t(dx)
\end{array}$$
and the statement would follow. However, \eqref{heuristic} was proved  only for $C^2$ functions that are constant outside a compact set (\cite[Lemma 3.1]{AngLorLun}), and there is no reason for $u_f(t, \cdot)$ be constant outside a compact set. So,  we   multiply by a sequence of cutoff functions that are equal to $1$ in $B_n$ and vanish outside $B_{2n}$. In this way we introduce extra terms; the further assumptions \eqref{hypo} will be used to get rid of such extra terms as $n\to \infty$. We state below the version of \eqref{heuristic} that we need here.

\begin{lemm}\label{derivative}
Under  Hypothesis \ref{base}, fix $[a,b]\subset I$. For every
$g\in C^{1,2}_b([a,b]\times\Rd)$ such that   $g(t,\cdot)$ is constant outside a compact set  for every $t\in [a,b]$,   the function $t \mapsto \int_{\Rd}g(t,x)\mu_t(dx)$ is continuously differentiable in $[a,b]$ and
\begin{eqnarray*}
\;\;\;\;\;\;\;\;\;\;\;\frac{d}{dt}\int_{\Rd}g(t,x) \mu_t(dx)=\int_{\Rd}D_tg(t,x)\mu_t(dx) -\int_{\Rd} (\mathcal{A}(t)g(t, \cdot))(x)\mu_t(dx),
\end{eqnarray*}
for every  $t \in [a,b]$.
\end{lemm}
\begin{proof}
The statement was proved in \cite[Lemma 3.1]{AngLorLun}, in the case of diffusion coefficients $q_{ij}$ depending only on $t$. But the proof relies on general properties of $G(t,s)$ that do not require this restrictive assumption, and can be followed word by word in our general context.
\end{proof}

\begin{thm}\label{martina}
Let  Hypotheses \ref{base}, \ref{smooth}, and \ref{hyp1_F} hold. In addition, we assume that for any  $s\in I$ and $\tau >s$ there exist three positive constants $C_i= C_i(s,\tau)$, $i=0,1,2$ such that
\begin{eqnarray}\label{hypo}
\left\{\begin{array}{lll}
|Q(t,x)x|\le C_0 |x|\varphi(x), \\[1mm]
{\rm{Tr}}(Q(t,x))\le C_1(1+|x|)\varphi(x),\\[1mm]
 \langle b(t,x),x \rangle \le C_2 |x|\varphi(x),
\end{array}\right.
\end{eqnarray}
\noindent
for any $t \in [s, \tau]$, $x\in\Rd$, where   $\varphi$ is the Lyapunov function introduced in Hypothesis \ref{base}(iii).
If there exists  $\psi_0\in \R$ such that \eqref{psi0} holds,
then for every $s\in I$ and for every  $f\in L^p(\Rd,\mu_t)$ we have
\begin{equation}\label{aim2}
\|u_f(t,\cdot)\|_{L^p(\Rd, \mu_t)}\le e^{\psi_0(t-s)}\|f\|_{L^p(\Rd,\mu_s)},\qquad\,\, t\ge s.
\end{equation}
In particular, if $\psi_0<0$ the null solution of \eqref{eq} is exponentially asymptotically stable in $L^p(\Rd,\mu_t)$.
\end{thm}
\begin{proof}
The proof is in two steps. In the first step we prove that \eqref{aim2} holds if $f\in C^1_b(\Rd)$. In this case
 $u_f$ is a classical solution and its space gradient is bounded, which helps to get rid of some of the extra terms obtained with the introduction of cutoff functions. In the second step, we consider any $f\in L^p(\Rd,\mu_s)$ and we prove the statement by an approximation procedure.

\vspace{2mm}
{\em Step 1.} Let $f \in C^1_b(\Rd)$, and let $u=u_f$ be the mild solution  to \eqref{p_nonlin}. $u$ is a classical solution by Theorem \ref{classical}, moreover \eqref{psi0} is stronger than \eqref{alge}; therefore by Theorem \ref{Exlarge} we have $I(f) = [s, +\infty)$.

To get $L^p$ estimates on $u$ we  introduce a sequence of cut-off functions $\theta_n$,   defined by
\begin{equation*}
\theta_n(x)=\zeta\left(\frac{|x|}{n}\right), \qquad\;\,x\in \R^d, \;\,n\in \N.
\end{equation*}
where $\zeta \in C^\infty(\R)$ is a nonincreasing function such that $\zeta(\xi)=1$ for $\xi\leq 1$, $\zeta(\xi)=0$ for $\xi\geq 2$.

In addition, since the term $|u|^{p-2}$ will appear in our computations,  to avoid unpleasant singularities in the case  $p <2$  we introduce the functions
$$v_{n,\varepsilon}:= (\theta_n u^2+\varepsilon)^{1/2},\quad\;\, n\in \N, \;\varepsilon>0. $$
We shall estimate the functions
\begin{equation}
\label{beta}
\beta_{n, \eps}(t):=  \|v_{n,\varepsilon}(t,  \cdot)\|_{L^p(\R^d, \mu_t)}, \quad t\geq s.
\end{equation}
To this aim we remark that $\beta_{n, \eps}$ is continuous in $[s, +\infty)$, for every $n\in \N$ and $\eps >0$. Indeed, for every $t$, $t_0\geq s$ we have
$$\begin{array}{l}
\ds \bigg| \int_{\R^d}  v_{n,\varepsilon}(t,  x)^p d\mu_t - \int_{\R^d} v_{n,\varepsilon}(t_0,  x)^p d\mu_{t_0}\bigg|
\\
\\
\leq \ds  \int_{\R^d} |v_{n,\varepsilon}(t,  x)^p - v_{n,\varepsilon}(t_0,x)^p| d\mu_t + \bigg|  \int_{\R^d} v_{n,\varepsilon}(t_0,x)^p d\mu_t -  \int_{\R^d} v_{n,\varepsilon}(t_0,x)^p d\mu_{t_0}\bigg|
\end{array}$$
The first term vanishes as $t\to t_0$ by the continuity of $v_{n,\varepsilon}$ and the Dominated Convergence Theorem, the second term vanishes as $t\to t_0$ by Lemma \ref{continuita'}.

By Lemma \ref{derivative},  $\beta_{n, \eps}$ is differentiable in $(s, +\infty)$. Let us  estimate its derivative.
For any $\tau >s  $, the function $v_{n,\varepsilon}$ satisfies
\begin{equation}\label{p_n}
\left\{
\begin{array}{ll}
D_t v =\ds  {\mathcal{A}}(t)v + \frac{u\theta_n}{v_{n, \eps}} \psi(t, u) + g_n ,\quad t \in (s,\tau ],
\\[1mm]
\medskip
v (s,  \cdot )= (\theta_n f^2+\varepsilon)^{1/2},
\end{array}\right.
\end{equation}
where
\begin{align}\label{estg_n}
g_n &: =  - \frac{u^2}{2v_{n,\varepsilon} } \mathcal{A}(\cdot)\theta_n  + \bigg( \frac{u^3\theta_n}{v_{n,\varepsilon}^3}-\frac{2u}{v_{n,\varepsilon}} \bigg) \langle Q(\cdot)\nabla \theta_n,\nabla_x u\rangle\nnm\\
&\quad-\frac{\theta_n}{v_{n,\varepsilon}} \bigg( 1 - \frac{u^2\theta_n}{v_{n,\varepsilon}^2}\bigg) \langle Q(\cdot)\nabla_x u,\nabla_x u\rangle+
\frac{u^4}{4v_{n,\varepsilon}^3}
\langle Q(\cdot) \nabla \theta_n, \nabla \theta_n \rangle
\nnm\\
& \leq  - \frac{u^2}{2v_{n,\varepsilon} } \mathcal{A}(\cdot)\theta_n  + \bigg( \frac{u^3\theta_n}{v_{n,\varepsilon}^3}-\frac{2u}{v_{n,\varepsilon}} \bigg) \langle Q(\cdot)\nabla \theta_n,\nabla_x u\rangle +
\frac{u^4}{4v_{n,\varepsilon}^3} \langle Q(\cdot) \nabla \theta_n, \nabla \theta_n \rangle \nnm\\
& := h_n
\end{align}
Multiplying the differential equation in \eqref{p_n} by $v_{n,\varepsilon}^{p-1}$ and using \eqref{estg_n} we get
\begin{align}\label{raffi}
v_{n,\varepsilon}^{p-1}D_t v_{n,\varepsilon} \leq v_{n,\varepsilon}^{p-1}{\mathcal{A}}(t)v_{n,\varepsilon}+ v_{n,\varepsilon}^{p-1}\frac{u\theta_n}{v_{n, \eps}} \psi(t, u) + v_{n,\varepsilon}^{p-1} h_n
\end{align}
for any $t \in (s, \tau ]$.
Now, by Theorem \ref{classical} ,   $\|\nabla_x u(t, \cdot)\|_\infty  \leq c_1$   and $\|D^2_x u(t, \cdot)\|_{L^\infty (B_R)} \le c_2(t-s)^{-1/2} $ for any $t \in (s, \tau ]$, $R>0$  and some positive constants $c_1$, $c_2$ independent of $t$. Using such  estimates and recalling that $\theta_n$ has compact support in $\Rd$, assumption \eqref{hypo} yields  that for any $t \in (s, \tau ]$, any $n \in \N$, the functions $v_{n,\varepsilon}^{p-1}(t)D_t v_{n,\varepsilon}(t)$, $v_{n,\varepsilon}^{p-1}{\mathcal{A}}(t)v_{n,\varepsilon}$ and
$h_{n} v_{n,\varepsilon}^{p-1}(t)$  belong to $L^1(\Rd, \mu_t)$.
Hence we can integrate  \eqref{raffi} with respect to $\mu_t$ in $\Rd$ to get
\begin{equation}
\label{primo}
\begin{array}{l}
\ds \int_{\Rd}D_t(v_{n,\varepsilon}^p) d\mu_t
\\
\\ \ds \le p \int_{\Rd}v_{n,\varepsilon}^{p-1}\mathcal{A}(t)v_{n,\varepsilon} d\mu_t+ p \int_{\Rd}v_{n,\varepsilon}^{p-2} u\theta_n \psi(t,u)\,d\mu_t
 + p \int_{\Rd}v_{n,\varepsilon}^{p-1} h_n\, d\mu_t,
 \end{array}
\end{equation}
for every $t \in (s,\tau ]$.   Lemma \ref{derivative},  applied to the function $g : = v_{n,\varepsilon}^p$ in any interval $[a,b]\subset (s, \tau]$,   gives
\begin{equation}
\label{secondo}
 \int_{\Rd}D_tv_{n,\varepsilon}^p d\mu_t= D_t \|v_{n,\varepsilon}\|_{L^p(\Rd,\mu_t)}^p+ \int_{\Rd}\mathcal{A}(t)(v_{n,\varepsilon}^p)d\mu_t, \quad s<t\leq \tau,
\end{equation}
and since
\begin{equation}
\label{fame}
 \mathcal{A}(t)(v_{n,\varepsilon}^p) = pv_{n,\varepsilon}^{p-1} \mathcal{A}(t)v_{n,\varepsilon}+p(p-1)v_{n,\varepsilon}^{p-2}\langle Q(t)\nabla_x v_{n,\varepsilon}, \nabla_x v_{n,\varepsilon}\rangle  ,
 \end{equation}
putting together \eqref{primo} and \eqref{secondo} we get
\begin{equation}
\label{stimaLp}
\begin{array}{lll}
 D_t \|v_{n,\varepsilon}(t, \cdot)\|_{L^p(\Rd,\mu_t)}^p & \le  & \ds p  \int_{\Rd}v_{n,\varepsilon}^{p-2}\theta_n u\psi(t,u) d\mu_t+
 \\
 \\
 && +\ds  p \int_{\Rd}(v_{n,\varepsilon}^{p-1}h_n
 -  (p-1) v_{n,\varepsilon}^{p-2}\langle Q(t)\nabla_x v_{n,\varepsilon}, \nabla_x v_{n,\varepsilon}\rangle ) d\mu_t
 \\
 \\
 & \leq &  \ds p \int_{\Rd}v_{n,\varepsilon}^{p-2}\theta_n u\psi(t,u) d\mu_t+ p \int_{\Rd} v_{n,\varepsilon}^{p-1}h_n  \, d\mu_t.
\end{array}
 \end{equation}
Now we claim that there exists $K>0$ such that
\begin{equation}
\label{estimate}
  \int_{\Rd}v_{n,\varepsilon}(t, \cdot)^{p-1}h_n(t, \cdot) d\mu_t \leq \frac{K}{n}, \quad s<t\leq \tau .
\end{equation}
Once  \eqref{estimate} is proved, assumption \eqref{psi0} allows us to estimate $v_{n,\varepsilon}^{p-2}\theta_n u \psi(t,u) $ by  $\psi_0 v_{n,\varepsilon}^{p-2}\theta_nu^2$  and then to proceed. Indeed, by \eqref{stimaLp} we obtain

\begin{equation}
\label{stimavneps}
\begin{array}{l}
\|v_{n,\varepsilon}(t_2, \cdot) \|_{L^p(\Rd,\mu_{t_2})} - \|v_{n,\varepsilon}(t_1, \cdot)\|_{L^p(\Rd,\mu_{t_1})}
\\
\\
\leq \psi_0\ds \int_{t_1}^{t_2} \bigg( \|v_{n,\varepsilon}(r, \cdot) \|_{L^p(\Rd,\mu_r)}^{ 1-p} \int_{\R^d} v_{n,\varepsilon}(r, \cdot)^{p-2}u^2(r,\cdot) \theta _n d\mu_r\bigg) dr
\\
\\
\ds + \int_{t_1}^{t_2} \bigg( \|v_{n,\varepsilon}(r, \cdot) \|_{L^p(\Rd,\mu_r)}^{ 1-p} \int_{\R^d} v_{n,\varepsilon}(r, \cdot)^{p-1}h_n(r, \cdot) \bigg) dr
 \end{array}
 \end{equation}
for any $s\leq  t_1 \leq t_2 \le \tau$. Hence, taking \eqref{estimate} into account and letting $n\to \infty$, we get
\begin{align}\label{parma2}
&\|(u(t_2, \cdot)^2+\varepsilon)^{1/2} \|_{L^p(\Rd,\mu_{t})} - \|(u(t_1, \cdot)^2+\varepsilon)^{1/2}\|_{L^p(\Rd,\mu_{s})}
\nnm\\
&\le \psi_0\ds \int_{t_1}^{t_2} \bigg( \|(u(r, \cdot)^2+\varepsilon)^{1/2} \|_{L^p(\Rd,\mu_r)}^{ 1-p} \int_{\R^d} (u(r, \cdot)^2+\varepsilon)^{(p-2)/2}u^2(r,\cdot) d\mu_r\bigg) dr
\end{align}

Letting $\varepsilon \to 0$ in \eqref{parma2} yields
\begin{align*}
 \|u(t_2, \cdot) \|_{L^p(\Rd,\mu_{t_2})}\leq & \|u(t_1, \cdot) \|_{L^p(\Rd,\mu_{t_1})} + \psi_0\int_{t_1}^{t_2} \|u(r, \cdot) \|_{L^p(\Rd,\mu_r)}dr
 \end{align*}
for any $s \le t_1 \le t_2 \le \tau$, and from the Gronwall Lemma, estimate \eqref{aim2} follows.

It remains to prove \eqref{estimate}. We have $\int_{\R^d} v_{n,\varepsilon}(t, \cdot)^{p-1}h_n(t, \cdot)d\mu_t = \sum_{k=1}^{3}I_k(t)$, where
$$I_1(t) = -\frac{1}{2}\int_{\R^d} v_{n,\varepsilon}(t, \cdot)^{p-2}u^2(t, \cdot)\mathcal{A}(t)\theta_nd\mu_t, $$
$$I_2(t) = \int_{\R^d} v_{n,\varepsilon}(t, \cdot)^{p-1} \bigg( \frac{u^3\theta_n}{v_{n,\varepsilon}^3}-\frac{2u}{v_{n,\varepsilon}} \bigg)  \langle Q(t, \cdot)\nabla \theta_n,\nabla_x u  \rangle d\mu_t, $$
$$I_3(t) = \int_{\R^d} v_{n,\varepsilon}(t, \cdot)^{p-4}   \frac{u^4}{4}   \langle Q(t, \cdot)\nabla \theta_n,\nabla \theta_n \rangle d\mu_t.  $$
Let us compute $\mathcal{A}(t)\theta_n$. For any $t \in I$ and $x \in \Rd\setminus\{0\}$, we have
\begin{align*}
{\rm Tr}(Q(t,x)D^2\theta_n(x)) =&\zeta''\left(\frac{|x|}{n}\right)\frac{\langle Q(t,x)x,x \rangle}{n^2|x|^2}
+\zeta'\left(\frac{|x|}{n}\right)\frac{\textrm{Tr}(Q(t,x))}{n |x|}\\
&- \zeta'\left(\frac{|x|}{n}\right)\frac{\langle Q(t,x)x,x \rangle}{n |x|^3 },
\end{align*}
and
\begin{equation*}
\langle b(t,x),\nabla \theta_n(x)\rangle =  \zeta'\left(\frac{|x|}{n}\right)\frac{\langle b(t,x),x\rangle}{n|x|}.
\end{equation*}
 Recalling that the supports of $\zeta'$ and $\zeta''$ are contained in $[1,2]$ and that $\zeta'\leq 0$,  \eqref{hypo} yields
$$\sup_{t\in[s,\tau]}|{\rm{Tr}}(Q(t,x)D^2 \theta_n(x))| \leq   \frac{C}{n} \varphi(x),\qquad\;\, x \in \Rd  ,$$
and
$$\inf_{t\in[s,\tau]}\langle b(t,x), \nabla \theta_n(x)\rangle \geq   \zeta'\left(\frac{|x|}{n}\right)\frac{C}{n} \varphi(x),\qquad\;\, x \in \Rd ,$$
where $C$ is a positive constant depending only on $C_0,C_1,C_2,\Vert \zeta'\Vert_\infty$ and $\Vert \zeta''\Vert_\infty$. Therefore, there is $K\geq 0$ such that
$\mathcal{A}(t)\theta_n\geq -K\varphi(x)/n$, that implies
$$I_1(t)\leq \frac{K}{2n} \int_{\R^d}v_{n,\varepsilon}(t, \cdot)^{p} \varphi \,d\mu_t \leq \frac{K}{2n}(\|f\|_{\infty}^2 +\eps)^{p/2}M_{\varphi}, $$
where $M_\varphi$ is the constant defined in \eqref{nor_fi}. In a similar way we estimate
$$|Q(t,x)\nabla \theta_n|\le  \Vert \zeta'\Vert_\infty \frac{|Q(t,x)x|}{n|x|}\leq \frac{C}{n}\varphi(x), \qquad\;\, t \in [s,\tau],\, x \in \Rd. $$
Since Hypothesis \ref{smooth} holds,   $\|\nabla_x u(t,\cdot)\|_{\infty}$ is bounded in $[s, \tau]$ by Theorem \ref{classical},
and
$$\begin{array}{lll}
|I_2(t)|& \leq & \ds \frac{3C}{n} \sup_{s\leq t\leq \tau}\|\nabla_x u(t,\cdot)\|_{\infty}\int_{\R^d} v_{n,\varepsilon}(t, \cdot)^{p-1} \varphi \,d\mu_t
\\
\\
&\leq & \ds \frac{3C}{n}\sup_{s\leq t\leq \tau}\|\nabla_x u(t,\cdot)\|_{\infty}(\|f\|_{\infty}^2 +\eps)^{(p-1)/2}M_{\varphi},
\end{array}$$
$$|I_3(t)|\leq \frac{C}{4n}\|\nabla \theta_n\|_{\infty} \int_{\R^d}  v_{n,\varepsilon}(t, \cdot)^{p} \varphi \,d\mu_t
\leq \frac{C}{4n^2}  \|\zeta'\|_{\infty} (\|f\|_{\infty}^2 +\eps)^{(p)/2}M_{\varphi} ,$$
and \eqref{estimate} follows.

\vspace{2mm}
{\em Step 2.} Now, let $f \in L^p(\Rd, \mu_s)$ and $(f_n)\subset C^1_b(\Rd)$  converge to $f$ in $L^p(\Rd,\mu_s)$ as $n \to +\infty$ (see \cite[Lemma 2.5.]{AngLorLun}). Step 1 yields
$$\| u_{f_n}(t,\cdot)\|_{L^p(\Rd,\mu_t)}\le e^{\psi_0(t-s)}\|f_n\|_{L^p(\Rd,\mu_s)},\qquad\;\, n\in \N,\,t\ge s.$$
Moreover, by estimate \eqref{terremoto} there is a constant $K$, depending only on $s$ and $\tau$, such that
$$\|u_{f_n}(t,\cdot)-u_f(t,\cdot)\|_{L^p(\Rd,\mu_t)}  \le K \|f_n-f\|_{L^p(\Rd, \mu_s)},\qquad\;\, s\leq t\leq \tau ,\, n \in \N .
$$
Consequently,
\begin{equation*}
\|u_f(t, \cdot)\|_{L^p(\Rd,\mu_t)} = \lim_{n\to \infty} \|u_{f_n}(t, \cdot)\|_{L^p(\Rd,\mu_t)}
\le e^{\psi_0(t-s)}\|f\|_{L^p(\Rd, \mu_s)},\quad s\leq t\leq \tau.
\end{equation*}
By the arbitrariness of $\tau >s$ we conclude that $u_f$ satisfies \eqref{aim2}.
\end{proof}

Now we turn to  hypercontractivity in   problem \eqref{p_nonlin}. As  in the linear case, we need some logarithmic Sobolev inequalities with respect to the measures $ \mu_t$.

\begin{hyp}\label{LSI}
There exists a positive constant $K$ such that
\begin{align}\label{moto}
\int_{\Rd} |g|^\gamma\log |g| \,d\mu_r\le  \|g\|_{L^\gamma(\Rd,\mu_r)}^\gamma \log\|g\|_{L^\gamma(\Rd,\mu_r)} + \gamma K\int_{\{g\neq 0\}}\!\!\!\!|g|^{\gamma-2}|\nabla g|^2 d\mu_r,
\end{align}
for any $r\in I$, $g\in C^1_b(\R^d)$ and $\gamma\in (1,+\infty)$.
\end{hyp}

\begin{thm}\label{martino}
Let the assumptions of Theorem  \ref{martina} be satisfied, and assume in addition that  Hypothesis  \ref{LSI} holds. For any $s\in I$, $p>1$ set
$$  p(t):=e^{\eta_0 K^{-1}(t-s)}(p-1)+1, \quad t\geq s, $$
where $\eta_0$ is the ellipticity constant of Hypothesis \ref{base}(ii),  and $K$ is the constant in \eqref{moto}.  Then for every $f\in L^p(\Rd, \mu_s)$,   $u_f(t, \cdot)\in L^{p(t)}(\Rd, \mu_t)$ for every $t>s$ and
\begin{equation}\label{claim}
\|u(t,\cdot)\|_{L^{p(t)}(\Rd, \mu_t)}\le e^{\psi_0(t-s)}  \|f\|_{L^p(\Rd, \mu_s)},\quad t>s.
\end{equation}
\end{thm}
\begin{proof}
We follow the proof of Theorem \ref{martina} and  we use the notation   introduced there; to shorten formulae we denote  the norm $\|\cdot \|_{L^{p(t)}(\R^d, \mu_t)}$ by $\|\cdot\|_{p(t)}$. As in Theorem \ref{martina}, in the first step we prove that \eqref{claim} holds for $f\in C^1_b(\R^d)$ and in the second step we consider any $f\in L^p(\Rd, \mu_s)$.

{\em Step 1.}
Let $f \in C^1_b (\Rd)$. The functions defined in \eqref{beta} are replaced here by
$$\beta_{n,\varepsilon}(t) := \|v_{n,\varepsilon}(t,  \cdot)\|_{L^{p(t)}(\R^d, \mu_t)}, \quad t\geq s. $$
Arguing as in Theorem \ref{martina} we see that $\beta_{n,\varepsilon}$ is continuous in $[s, +\infty)$.   Lemma \ref{derivative} yields that the function $t\mapsto \int_{\Rd} v_{n,\varepsilon}^{p(t)}d\mu_t$
is differentiable in $(s,+\infty)$, for any $n \in \N$, $\varepsilon>0$, and   using \eqref{fame} we get
$$
\begin{array}{l}
\ds \int_{\Rd}D_t v_{n,\varepsilon}^{p(t)}d\mu_t-\int_{\Rd}\A(t)v_{n,\varepsilon}^{p(t)}d\mu_t
=  \ds p'(t)\int_{\Rd}  v_{n,\varepsilon}^{p(t)}\log v_{n,\varepsilon}d\mu_t
\\
\\
 \ds +p(t)\int_{\Rd} v_{n,\varepsilon}^{p(t)-1}g_n(u)d\mu_t
 - p(t)(p(t)-1)\int_{\Rd}  v_{n,\varepsilon}^{p(t)-2}\langle Q(t)\nabla_x v_{n,\varepsilon},\nabla_x v_{n,\varepsilon}\rangle d\mu_t.
\end{array}$$
Therefore, $\beta_{n,\varepsilon}$
is differentiable in $(s,+\infty)$,  and its derivative is given by
$$\begin{array}{lll}
\beta_{n,\varepsilon}'(t) &= & \ds  \|v_{n,\varepsilon}(t,\cdot)\|_{p(t)}\left\{ -\frac{p'(t)}{p^2(t)}\log \int_{\Rd}v_{n,\varepsilon}^{p(t)}d\mu_t\right.
\\
\\
&& \ds +\frac{1}{p(t)\|v_{n,\varepsilon}(t,\cdot)\|_{p(t)}^{p(t)}}\left[ p'(t)\int_{\Rd}   v_{n,\varepsilon}^{p(t)}\log v_{n,\varepsilon}d\mu_t +p(t)\int_{\Rd} v_{n,\varepsilon}^{p(t)-1}g_n d\mu_t \right.
\\
\\
&& \ds \left. \left. - p(t)(p(t)-1)\int_{\Rd}  v_{n,\varepsilon}^{p(t)-2}\langle Q(t, \cdot)\nabla_x v_{n,\varepsilon},\nabla_x v_{n,\varepsilon}\rangle d\mu_t  \right] \right\}.
\end{array}$$
By the logarithmic Sobolev inequality \eqref{LSI} and the ellipticity condition,
$$
\begin{array}{l}
\ds p'(t)\int_{\Rd}   v_{n,\varepsilon}^{p(t)}\log v_{n,\varepsilon}d\mu_t - p(t)(p(t)-1)\int_{\Rd}  v_{n,\varepsilon}^{p(t)-2}\langle Q(t, \cdot)\nabla_x v_{n,\varepsilon},\nabla_x v_{n,\varepsilon}\rangle d\mu_t
\\
\\
\leq \ds p'(t) \|v_{n,\varepsilon}(t,\cdot)\|_{p(t)}^{p(t)} \log \|v_{n,\varepsilon}(t,\cdot)\|_{p(t)}
\\
\\
\ds +
p(t)( p'(t) K -\eta_0 (p(t)-1))\int_{\Rd}  v_{n,\varepsilon}^{p(t)}|\nabla_x v_{n,\varepsilon}|^2 d\mu_t
\\
\\
= \ds p'(t) \|v_{n,\varepsilon}(t,\cdot)\|_{p(t)}^{p(t)} \log \|v_{n,\varepsilon}(t,\cdot)\|_{p(t)}.
\end{array}$$
Such inequality, together with
the dissipativity condition \eqref{psi0} and the inequality $g_n\leq h_n$ (see \eqref{estg_n}) yields
$$\begin{array}{l}
D_t\|v_{n,\varepsilon}(t,\cdot)\|_{p(t)} \le
\\
\\
\leq \ds \|v_{n,\varepsilon}(t,\cdot)\|_{p(t)}^{1-p(t)} \bigg(
\int_{\Rd}v_{n,\varepsilon}^{p(t)-2}\theta_n u\psi(t,u) d\mu_t+   \int_{\Rd} v_{n,\varepsilon}^{p(t)-1}h_n  \, d\mu_t\bigg)
\end{array}$$
which  is equivalent to \eqref{stimaLp} with $p(t)$ in place of $p$. It implies that \eqref{stimavneps} holds, still with $p(t)$ in place of $p$, and arguing as in the proof of Theorem \ref{martina} we arrive at \eqref{claim}.

{\em Step 2.}  Let $f\in L^p(\Rd, \mu_s)$,  and let $(f_n)$ be a sequence of functions in $C^1_b(\Rd)$ such that $\|f_n-f\|_{L^p(\Rd,\mu_s)}$ vanishes as $n\to \infty$.
From Step 1, applied to the linear case $\psi\equiv 0$ we obtain  that $G(t, r)$ maps $L^{p( r)}(\Rd,\mu_r)$ into $L^{p(t)}(\Rd,\mu_t)$ for $t\geq r\geq s$, and
$$\|G(t, r)g\|_{L^{p(t)}(\Rd,\mu_t)}\le e^{\psi_0(t-r)}\|g\|_{L^{p( r)}(\Rd,\mu_r)}, \quad t\ge r\geq s, \;g\in L^{p( r)}(\Rd,\mu_r). $$
Fix $\tau >s$. According to Hypothesis \ref{hyp1_F}, let $L>0$ be such that $|\psi(r, x)-\psi(r,y)|\leq L|x-y|$ for every $r\in[s, \tau]$, $x$, $y\in \R^d$. Using \eqref{defi_mild} we obtain for $n$, $m\in \N$ and $s\leq t\leq \tau$
$$\begin{array}{l}
\|u_{f_n}(t,\cdot)-u_{f_m}(t,\cdot)\|_{ p(t) }
\le   e^{\psi_0(t-s)} \|f_n-f_m\|_{L^p(\Rd, \mu_s)}
\\
\\ \ds + \int_s^t e^{\psi_0(t-r)} \|\psi(r, u_{f_n}(r, \cdot))-\psi(r, u_{f_m}(r,\cdot))\|_{p(r )}dr
\\
\le \ds e^{\psi_0(t-s)} \Big(\|f_n-f_m\|_{L^p(\Rd, \mu_s)} + L\int_s^t e^{\psi_0(t-r)}\|u_{f_n}(r,\cdot)-u_{f_m}(r,\cdot)\|_{p( r)} dr\Big)
\end{array}$$
The Gronwall Lemma yields
$$\|u_{f_n}(t,\cdot)-u_{f_m}(t,\cdot)\|_{L^{p(t)}(\Rd,\mu_t)}\le e^{(\psi_0+L)(t-s)}\|f_n-f_m\|_{L^p(\Rd, \mu_s)}, \quad s\leq t\leq \tau, $$
so  that $(u_{f_n}(t,\cdot))$ is a Cauchy sequence in $L^{p(t)}(\Rd, \mu_t)$ for any $t \in [s,\tau]$, and it converges to some $v(t) \in L^{p(t)}(\Rd, \mu_t)$. We already know, from estimate \eqref{terremoto}, that  $(u_{f_n}(t,\cdot))$ converges to $u_f (t, \cdot)$ in $L^{p }(\Rd, \mu_t)$ for any $t \in [s,\tau]$. Therefore, $v(t) = u_f (t, \cdot) \in L^{p(t)}(\Rd, \mu_t)$ for any $t \in [s,\tau]$. Moreover, by Step 1 we have
$$\|u_{f_n}(t, \cdot)\|_{L^{p(t)}(\Rd,\mu_t)}\le e^{\psi_0(t-s)}\| f_n \|_{L^{p }(\Rd,\mu_s)}, \quad t\geq s, \;n \in \N, $$
and letting $n\to \infty$ we obtain $\|u_{f}(t,\cdot)\|_{L^{p(t)}(\Rd,\mu_t)}\le e^{\psi_0(t-s)}\|f\|_{L^p(\Rd,\mu_s)}$ for any $t \in [s,\tau]$,
which yields \eqref{claim} since $\tau$ is arbitrary.
\end{proof}

\begin{rmk}{\rm{
Assumptions \eqref{hypo} are not very restrictive, because in explicit examples we can play with the choice of $\varphi$. For instance, let $\A(t)$ be as in \eqref{oper} with
$$Q(t,x)=q(t)(1+|x|^2)^lQ^0,\quad b(t,x)=-b(t)x(1+|x|^2)^m,\qquad\;\, t\in I, x \in \Rd,$$
with $m$, $l \geq 0$. Here $Q^0$ is a positive definite real symmetric matrix and the functions $q$, $b$ have positive infimum and belong to $C^{\alpha/2}_{\rm loc}(I)\cap C_b(I)$. A straightforward computation shows that for every $r>0$ the function $\varphi(x)= (1+|x|^2)^r$, $x \in \Rd$, $r>0$ satisfies
\begin{align*}
(\A(t)\varphi)(x) \le  &2r\varphi(x)\left\{\sup_{t\in I} q(t)[2(r-1)^+\langle Q_0x, x\rangle (1+|x|^2)^{l-2}+{\rm Tr}(Q^0)(1+|x|^2)^{l-1}]\right.\nnm\\
&\left.\quad\quad\quad\,\,- \inf_{t\in I}b(t)|x|^2(1+|x|^2)^{m-1}\right\}.
\end{align*}
Thus, if $m>l-1$ Hypothesis \ref{base}(iii) and Hypothesis \ref{smooth} are satisfied, for every choice of $r>0$. Assumptions \eqref{hypo} hold provided $r$ is chosen large enough ($r\geq l$). }}\end{rmk}

 %%%%%%%%%%%%%%%%%%%%%%%%%%%%%%%%%%%%%%%%%%%%%%%%%%%%%%%%%%%%%%%%%%%%%%%%%

\appendix
\section{Linear parabolic equations in balls}

This Appendix is devoted to the proof of estimates \eqref{U1}. Since $R$ is arbitrary, everywhere we replace $R+1$ by $R$.
Our tools are the general results of \cite{Acq,AT} and interpolation arguments.

We choose $X= C(\overline{B}_R)$. The realizations $A(t): D(A(t))=\{ f\in \cap_{p>1}W^{2,p}(B_R):\;f_{|\partial B_R}=0, \;\mathcal{A}(t)f\in X\}$ of
$\mathcal{A}(t)$ in $X$ are sectorial operators by the Stewart's Theorem (\cite{Stewart}). Their domains depend on $t$, but the interpolation spaces $(X, D(A(t)))_{\theta, \infty}$  are independent of $t$. Indeed, we have
\begin{equation}
\label{interp1}
(X, D(A(t)))_{\theta, \infty}=  \{ f\in C^{2\theta}(\overline{B}_R):\; f_{|\partial B_R}=0\}
\end{equation}
for $\theta\neq 1/2$ and
\begin{equation*}
(X, D(A(t)))_{1/2, \infty}=  \{ f\in \mathcal{C}^1(\overline{B}_R):\; f_{|\partial B_R}=0\}
\end{equation*}
where $\mathcal{C}^1(\overline{B}_R) $ is the Zygmund space of the continuous functions such that
$$\sup_{x, y\in \overline{B}_R, \,x\neq y} \frac{|f(x) + f(y) -2f((x+y)/2)|}{|x-y|} <+\infty .$$
The respective norms are equivalent to the $(X, D(A(t)))_{\theta, \infty}$-norm, with equivalence constants depending only on $R$ and on the H\"older norm of the coefficients in $[a,b]\times B_R$. Moreover,  for $0<\theta\leq \alpha/2$,
\begin{equation}
\label{interp3}
\begin{array}{l}
\{ f\in D(A(t)):\; A(t)f\in (X, D(A(t)))_{\theta, \infty}\} =
\\
\\= \{ f\in C^{2\theta + 2}(\overline{B}_R):\;  f_{|\partial B_R}=A(t)f_{|\partial B_R} =0\}
\end{array}
\end{equation}
and in such space the $C^{2\theta + 2}$ norm is equivalent to $f\mapsto \|f\|_{\infty} + \|A(t)f\|_{(X, D(A(t)))_{\theta, \infty}}$, with equivalence constants independent of $t$.
See \cite[Sect. 3.1.5]{libro}.

Moreover, each $A(t)$ is one to one, and for fixed $\theta\in (0, \alpha]$ we have
$$\|(A(t)^{-1}-A(s)^{-1})f\|_{C^{ \theta}(\overline{B}_R)}\leq C(t-s)^{\alpha/2}\|f\|_{\infty}, \quad a\leq s\leq t\leq b, \;f\in X, $$
as easily seen using the Schauder estimates for elliptic equations with $\alpha$-H\"older continuous coefficients.

Therefore, Hypothesis 7.3 of \cite{AT} is satisfied. By \cite[Thm. 4.2(iii)]{Acq}, for every $\mu \in (0,1)$ there exists $C_1= C_1(\mu )>0$ such that for every $f\in (X, D(A(0)))_{\mu, \infty}$ we have
$$ \begin{array}{l}
\| G(t,s)f\|_{(X, D(A(t)))_{\mu, \infty}} \leq C\|f\|_{(X, D(A(0)))_{ \mu, \infty}}, \quad a\leq s<t\leq b,
\\
\\
(t-s)^{ \mu}\| A(t)G(t,s)f\|_{X} \leq C\|f\|_{(X, D(A(0)))_{ \mu, \infty}}, \quad a\leq s<t\leq b.
\end{array}$$
Such estimates hold also for $\mu=0$, with the convention $(X, D(A(0)))_{0, \infty}=X$, by \cite[Thm. 4.1(i)]{Acq}.
By interpolation, for every $\mu \in [0,1)$
and
$\beta \in (0,1)$ there exists $C_2= C_2(\mu, \beta)>0$ such that for every $f\in (X, D(A(0)))_{\mu, \infty}$  we have
\begin{equation}
\label{intermedia}
(t-s)^{ \max\{  \beta-\mu, 0\}} \| G(t,s)f\|_{(X, D(A(t)))_{\beta, \infty}} \leq C_2\|f\|_{(X, D(A(0)))_{ \mu, \infty}}, \quad a\leq s<t\leq b.
\end{equation}
Taking $\beta = \eta/2$ and $\mu=0$,  \eqref{U1} follows for $\eta<2$, $\eta \neq 1$ from the characterizations \eqref{interp1}.
For $\eta =1$ such arguments give an estimate only  for   $\|G(t,s)f\|_{{\mathcal C}^{1}(\overline{B}_R)}$; however the
estimate in the $C^1$ norm is readily recovered from the estimates for $\eta \neq 1$ by interpolation, using e.g. the interpolatory estimate
$$\|\varphi\|_{C^1(\overline{B}_R)} \leq C\|\varphi\|_{C^{1/2}(\overline{B}_R)}^{1/2} \|\varphi\|_{C^{3/2}(\overline{B}_R)}^{1/2} , \quad \varphi\in C^{3/2} (\overline{B}_R). $$
So,   \eqref{U1} holds for $0<\eta <2$. Now we prove  \eqref{U1} for $\eta \in (2, 2+\alpha]$.

By  \cite[Thm. 6.4]{AT}, for every $\mu \in (0,1)$ and for every $\beta \leq \alpha/2$ there exist  $C_3= C_3(\mu)$ and $C_4=C_4(\mu, \beta)>0$ such that for every $f\in (X, D(A(0)))_{\mu, \infty}$
$$
\begin{array}{l}
 (t-s)^{1-\mu }\|A(t)G(t,s)f\|_{X} \leq C_3 \|f\|_{(X, D(A(0)))_{ \mu, \infty}}, \quad a\leq s<t\leq b,
\\
\\
 (t-s)^{1-\mu +\beta} \|A(t)G(t,s)f\|_{(X, D(A(t)))_{\beta, \infty}} \leq C_4\|f\|_{(X, D(A(0)))_{ \mu, \infty}}, \quad a\leq s<t\leq b.
\end{array}
$$
Taking into account   \eqref{interp3},   such estimates (with $\mu=\theta/2$, $\beta = (\eta-2)/2 $) yield
$$(t-s)^{ (\eta-\theta) /2}\|G(t, s)f\|_{C^{ \eta}(\overline{B}_R)} \leq C_5\|f\|_{C^{\theta}(\overline{B}_R)}, \quad a\leq s<t\leq b, $$
and \eqref{U1} is proved for $\eta\in (2, 2+\alpha]$.

As in the case $\eta=1$, a direct use of \eqref{intermedia} with $\mu=\theta/2$, $\beta=1$,  does not give an  estimate for $\|G(t,s)f\|_{C^{2 }(\overline{B}_R)}$ since the graph norm of each $A(t)$ is weaker than the $C^2$ norm. However,   as before we recover the $C^2$ estimate using the interpolatory estimate
$$\|\varphi\|_{C^2(\overline{B}_R)} \leq C\|\varphi\|_{C^{2-\varepsilon}(\overline{B}_R)}^{1/2} \|\varphi\|_{C^{2+\varepsilon}(\overline{B}_R)}^{1/2} , \quad \varphi\in C^{2+\varepsilon} (\overline{B}_R), $$
and  \eqref{U1} with $\eta = 2-\varepsilon$, $\eta = 2+\varepsilon$, $ \varepsilon \in (0, \alpha)$.

\end{document}